\documentclass[twoside,12pt]{article}

\usepackage{amssymb, amsmath}
\usepackage{mathtools}
\usepackage[thmmarks,amsmath]{ntheorem}%[thmmarks,amsmath,standard]{ntheorem}

\usepackage[OT1]{fontenc}

\usepackage[utf8]{inputenc}
\usepackage[T1]{fontenc}
\usepackage{dsfont}% fuer \mathds{R} (besser als \mathbb{R})
\usepackage{latexsym} % symbols, arrows

\usepackage{enumerate}	% labels in enumerate environments (better) \begin{enumerate}[(I.)]
\usepackage{footnpag} % Nummerierung von Fussnoten auf jeder Seite bei 1 beginnen
\usepackage{makeidx}  % Index

\usepackage{geometry}	%Formatierung Seiten
\usepackage{parcolumns}

\usepackage{graphicx}
\usepackage{color} 

\usepackage[all]{xy}

\usepackage{url}

\renewcommand{\geq}{\geqslant}	% schöne Ungleichheitszeichen!
\renewcommand{\leq}{\leqslant}
\def\mid|{\hs\middle|\hs}
\def\hs{\hspace*{0.1cm}}

\def\1{\mathbb{1}}
\newcommand{\sub}{\subseteq}	
\renewcommand{\1}{\mathds{1}}					% charakteristische Fkt

\newcommand{\B}{\mathcal{B}}

\newcommand{\I}{\mathcal{I}}

\newcommand{\NN}{\mathbb{N}}	% natürl. Zahlen
\newcommand{\RR}{\mathbb{R}} % reelle Zahlen
\def\t{\textnormal}

\def\c{\cite }
\newcommand{\hide}[1]{}  %This operator forgets what is written in { }%
\def\phi{\varphi}
\def\rho{\varrho}

\DeclareMathAlphabet\mathbfcal{OMS}{cmsy}{b}{n}

\newcommand{\solv}{\mathop{\textnormal{solv}}}

\def\dd{^\textnormal{dd}}
\newcommand{\Sup}{\vee}		% sign for Supremum
\let\int\relax % to avoid a "spurious" error message
\DeclareMathOperator{\int}{int}

\renewcommand{\theequation}{\arabic{equation}}

\usepackage{enumitem}
\setenumerate[0]{label=(\roman*)}
\setenumerate[2]{label=(\alph*)}

\newcommand*{\smallmath}[2][4]{\scalebox{#1}{$#2$}}%
\newcounter{Zaehler}
\theoremstyle{plain}
\theoremheaderfont{\normalfont\bfseries}
\theorembodyfont{\itshape}
\theoremseparator{.} % e.g. Lemma 1.
\theorempreskip{\topsep}
\theorempostskip{\topsep} %\topsep refers to the space that LATEX inserts above and below lists
\theoremnumbering{arabic}
\theoremsymbol{}

\newtheorem{theorem}{Theorem}%[Zaehler]	%<-- Format: Satz a.
\newtheorem{lemma}[theorem]{Lemma}
\newtheorem{corollary}[theorem]{Corollary}
\newtheorem{proposition}[theorem]{Proposition}
\theoremstyle{plain}
\theoremheaderfont{\normalfont\bfseries}
\theorembodyfont{\upshape}
\theoremseparator{.} % e.g. Lemma 1.
\theorempreskip{\topsep}
\theorempostskip{\topsep} %\topsep refers to the space that LATEX inserts above and below lists
\theoremnumbering{arabic}
\theoremsymbol{}

\newtheorem{example}[theorem]{Example}
\newtheorem{numremark}[theorem]{Remark}%[theorem]{Remark}
\makeatletter
\newcommand{\eqnum}{\leavevmode\hfill\refstepcounter{equation}\textup{\tagform@{\theequation}}}
\makeatother

\theoremstyle{nonumberplain}
\theoremheaderfont{\itshape}
\theorembodyfont{\normalfont}
\theoremseparator{.} % e.g. Proof.
\theorempreskip{\topsep}
\theorempostskip{\topsep} %\topsep refers to the space that LATEX inserts above and below lists
\theoremsymbol{\ensuremath\square} % at the end of the Theorem
\qedsymbol{$\Box$}

\newtheorem{proof}{Proof}

\begin{document}
\title{Vector lattice covers of ideals and bands in pre-Riesz spaces}
\author{Anke Kalauch\thanks{Technische Universit\"{a}t Dresden,  Technische Universität Dresden, Institut f\"{u}r Analysis, 01062 Dresden, Germany; anke.kalauch@tu-dresden.de.}, Helena Malinowski\thanks{Unit for BMI, North-West University, Private Bag X6001-209, Potchefstroom 2520, South Africa; lenamalinowski@gmx.de}}
%\date{}
\maketitle
\begin{abstract}
Pre-Riesz spaces are ordered vector spaces which can be order densely embedded into vector lattices, their so-called vector lattice covers. Given a vector lattice cover $Y$ for a pre-Riesz space $X$, we address the question how to find vector lattice covers for subspaces of $X$, such as ideals and bands. We provide conditions such that for a directed ideal $I$ in $X$ its smallest extension ideal in $Y$ is a vector lattice cover. We show a criterion for bands in $X$ and their extension bands in $Y$ as well. Moreover, we state properties of ideals and bands in $X$ which are generated by sets, and of their extensions in $Y$.  
\end{abstract}

%\textbf{Keywords:} pre-Riesz space, ordered vector space, vector lattice cover, order ideal, band, extension ideal, extension band

%\textbf{Mathematics Subject Classification:} 46A40; 06F20

%%%%%%%%%%%%%%%%%%%%%%%%%%%%%%%%%%%%%%%%%%%%%%%%%%%%%%%%%
%%%%%%%%%%%%%%%%%%%%%%%%%%%%%%%%%%%%%%%%%%%%%%%%%%%%%%%%%
%%%%
%%%%				CONTENT
%%%%
%%%%%%%%%%%%%%%%%%%%%%%%%%%%%%%%%%%%%%%%%%%%%%%%%%%%%%%%%
%%%%%%%%%%%%%%%%%%%%%%%%%%%%%%%%%%%%%%%%%%%%%%%%%%%%%%%%%

%%%%%%%%%%%%%%%%%%%%%%%%%%%%%%%%%%%%%%%%%%%%%%%%%%%%%%%%%
%%%%
%%%%				Intro
%%%%
%%%%%%%%%%%%%%%%%%%%%%%%%%%%%%%%%%%%%%%%%%%%%%%%%%%%%%%%%
\section{Introduction}
In the analysis of partially ordered vector spaces,  subspaces as ideals and bands play a central role. In vector lattices, problems involving those subspaces are broadly discussed in the literature, see e.g.\ \cite{PosOp,Zaa1,Vulikh_en}.
Directed ideals in partially ordered vector spaces were introduced in 
\cite[Definition 2.2]{Kad1951} 
and studied in 
\cite{Bons1954,Bons1956a,Kist1961,Fuc1966}; for a more recent overview see \cite{7}. The investigation of disjointness and bands in partially ordered vector spaces starts in \cite{1}, where for the definition of disjointness sets of upper bounds are used instead of lattice operations. Here a band is defined to be a set that equals its double-disjoint complement, analogously to the notion in Archimedean vector lattices. Ideals and bands are mostly considered in pre-Riesz spaces. An ordered vector space $X$ is called \emph{pre-Riesz} if there is a vector lattice $Y$ and a bipositive linear map $i\colon X\to Y$ such that $i(X)$ is order dense in $Y$. The pair $(Y,i)$ is then called a \emph{vector lattice cover} of $X$. The theory of pre-Riesz spaces and their  vector lattice covers is due to
van Haandel, see 
\cite{vanHaa}. Pre-Riesz spaces cover a wide range of examples, in particular note that every Archimedean directed ordered vector space is a pre-Riesz space. In pre-Riesz spaces, every band is a solvex ideal, see \cite{2}. However, bands are not directed, in general.

In certain cases, ideals and bands in pre-Riesz spaces correspond to ideals and bands in a vector lattice cover, see \cite{3,2}. More precisely, for every solvex ideal $I$ in $X$ there is a smallest ideal $\hat{I}$ in $Y$ that extends $I$, i.e.\ the pre-image of $\hat{I}$ under $i$ equals $I$. This applies, in particular, to directed ideals, as every directed ideal is solvex, see \cite{7}. Furthermore, for every band   
$B$ in $X$ there is a smallest band $\hat{B}$ in $Y$ that extends $B$.
Conversely, ideals in $Y$ can be restricted to ideals in $X$,
and in fordable pre-Riesz spaces bands in $Y$ can be restricted to bands in $X$. Note that fordability ensures that there are sufficiently many disjoint elements in $X$. 
 The extension and restriction method is a central tool to obtain properties of ideals and bands in pre-Riesz spaces, see \cite{2,6}. 

In the present paper we answer more specific questions on extension of ideals and bands generated by sets. Furthermore, we deal with the problem under which conditions for an ideal $I$ in $X$ the set $i(I)$ is order dense in $\hat{I}$, which means that $(\hat{I},i|_I)$ is a vector lattice cover of $I$. The analogous problem for bands is studied as well. These investigations are motivated by research on spaces of operators on vector lattices. 
In many cases, such an operator space is Archimedean and directed and, hence, pre-Riesz, but not a vector lattice, such that one would like to construct an appropriate vector lattice cover. The results of the present paper can be seen as a fundament for these further studies on spaces of operators.     

The paper is organized as follows. In Section 2 we collect all preliminaries.
Section 3 is dedicated to properties of directed ideals generated by a set, where we show the following: 
If $S$ is a non-empty subset of $X_+$ and $I$ the ideal generated by $S$, then $\hat{I}$ equals the ideal in $Y$ generated by $i(S)$. As a consequence, we obtain a
characterization of directed ideals by means of extension ideals. 
For the  band $B$ generated by $S\sub X$, we prove that  $\hat{B}$ equals the band in $Y$ generated by $i(S)$, provided  $X$ is fordable. 
In the last statement fordability can not be omitted, as we see in an example.

Section 4 contains the main results of this paper. The aim is to provide a technique to compute vector lattice covers of ideals and bands in certain pre-Riesz spaces. We show that for a directed ideal $I$ the set $i(I)$ is order dense in $\hat{I}$, provided $X$ is pervasive. Under this condition, we establish that for a band $B$ in $X$ the set 
$i(B)$ is order dense in $\hat{B}$ if and only if $i(B)$ is majorizing in $\hat{B}$.
These statements also allow to compute 
the Dedekind completions of $I$ and $B$ as subspaces of the Dedekind completion of $X$.

The paper contains several examples that clarify the conditions in the theorems and show the limitations of the technique developed in Section 4. We conclude the paper by a surprising example which shows that even for a directed band $B$ in $X$ the smallest extension ideal and the smallest extension band of $B$ might or might not coincide, depending on the choice of the vector lattice cover.

%%%%%%%%%%%%%%%%%%%%%%%%%%%%%%%%%%%%%%%%%%%%%%%%%%%%%%%%%
%%%%
%%%%
%%%%				Prelims
%%%%
%%%%
%%%%%%%%%%%%%%%%%%%%%%%%%%%%%%%%%%%%%%%%%%%%%%%%%%%%%%%%%
\section{Preliminaries}
Let $X$ be a real vector space and let $X_+$ be a \emph{cone} in $X$, that is,  $X_+$ is a wedge ($x,y\in X_+$ and $\lambda,\mu\geq 0$ imply $\lambda x + \mu y \in X_+$) and $X_+ \cap (-X_+) =\left\{0\right\}$. In $X$ a partial order is defined by $x\leq y$ whenever $y-x\in X_+$. 
%We denote the set of all positive elements in $X$ by $X_+$. The set $X_+$ is a cone in $X$. Every cone $K$ in $X$ induces a partial order such that $X_+=K$.
% The pair $(X,\leq)$ 
The space $(X,X_+)$ (or, loosely $X$) is then called a (\emph{partially}) \emph{ordered vector space}. %We write loosely $X$ instead of $(X,\leq)$, provided that $X_+$ or the partial order is fixed in advance. 
For a linear subspace $D$ of $X$ we
consider in $D$ the order induced from $X$, i.e.\ we set $D_+:=D\cap X_+$.

An ordered vector space $X$ is called \emph{Archimedean} if for every $x,y\in X$ with $nx\leq y$ for every $n\in\NN$ one has $x\leq 0$. 
Clearly, every subspace of an Archimedean ordered vector space is Archimedean.
A subspace $D\sub X$ is called \emph{directed} if for every $x,y\in D$ there is an element $z\in D$ such that $x\leq z$ and $y\leq z$. An ordered vector space $X$ is directed if and only if  $X_+$ is \emph{generating} in $X$, that is, $X = X_+ - X_+$. A linear subspace $D$ of $X$ is \emph{majorizing} in $X$ if for every $x\in X$ there exists $d\in D$ with $x\leq d$. The ordered vector space $X$ has the \emph{Riesz decomposition property} (\emph{RDP}) if for every $x_1,x_2,z \in X_+$ with $z \leq x_1+x_2$ there exist $z_1,z_2 \in X_+$ such that $z = z_1+z_2$ with $z_1 \leq x_1$ and $z_2\leq x_2$. The space $X$ has the RDP if and only if for every $x_1,x_2,x_3,x_4\in X$ with $x_1,x_2 \leq x_3,x_4$ there exists $z\in X$ such that $x_1, x_2 \leq z \leq x_3, x_4$. %As usual, a set $M\sub X$ is called (\emph{order}) \emph{bounded above}, if there exists $z\in X$ such that for every $x\in M$ one has $x\leq z$. Analogously, \emph{bounded below} is defined, and $M$ is called (\emph{order}) \emph{bounded}, if it is bounded above and below.

For standard notations in the case that $X$ is a vector lattice see \c{PosOp}. Recall that a vector lattice $X$ is \emph{Dedekind complete} if every non-empty subset of $X$ that is bounded above has a supremum. 
%By a subspace of an ordered vector space or a vector lattice we mean an arbitrary linear subspace with the inherited order. We do not require it to be a lattice or a sublattice. 
We say that a linear subspace $D$ of a vector lattice $X$ \emph{generates $X$ as a vector lattice} if for every $x\in X$ there exist $a_1, \ldots, a_n, b_1, \ldots, b_m \in D$ such that $x = \bigvee_{i=1}^{n}a_i - \bigvee_{j=1}^{m}b_j$.

We call a linear subspace $D$ of an ordered vector space $X$ \emph{order dense} in $X$ if for every $x\in X$ we have \[x = \inf\left\{z\in D \mid| x\leq z\right\},\] that is, the greatest lower bound of the set $\left\{z\in D \mid| x\leq z\right\}$ exists in $X$ and equals $x$, see \c[p.~360]{113}. Clearly, if $D$ is order dense in $X$, then $D$ is majorizing in $X$.
%%%%%%%%%%%%%%%%%%%%%%%%%%%%%%%%%%%%%%%%%%%%%%%%%%%%%%%%%
%%%%	BEGIN
%%%%%%%%%%%%%%%%%%%%%%%%%%%%%%%%%%%%%%%%%%%%%%%%%%%%%%%%%
We give a characterization of order denseness in Lemma \ref{properties.29} below. For this, we need the following preliminary result.

\begin{lemma}\label{properties.27.5}
	Let $X$ be an ordered vector space. Let $A \sub B \sub X$ and $\sup A = y\in X$. Moreover, let $b\leq y$ be true for every $b\in B$. Then the supremum $\sup B$ exists and we have $\sup B = y$.
	Analogously, for $A \sub B \sub X$ with $y=\inf A\in X$ and $y \leq b$ for every $b\in B$ it follows $\inf B = y$.
\end{lemma}
%%%%%%%%%%%%%%%%%%%%%%%%%%%%%%%%%%%%%%%%%%%%%%%%%%%%%%%%%
%%%%	END
%%%%%%%%%%%%%%%%%%%%%%%%%%%%%%%%%%%%%%%%%%%%%%%%%%%%%%%%%
\begin{proof}
	Since we have $b \leq y$ for every $b\in B$, the element $y$ is an upper bound of $B$. Let $s\in X$ be another upper bound of $B$. Due to $A\sub B$ and $b \leq s$ for every $b\in B$ it follows $a \leq s$ for every $a\in A$. Thus we obtain $\sup A = y \leq s$. Hence $y=\sup B$.
	The statement for the infimum follows analogously.
\end{proof}

%%%%%%%%%%%%%%%%%%%%%%%%%%%%%%%%%%%%%%%%%%%%%%%%%%%%%%%%%
%%%%	BEGIN
%%%%%%%%%%%%%%%%%%%%%%%%%%%%%%%%%%%%%%%%%%%%%%%%%%%%%%%%%
\begin{lemma}\label{properties.29}
	Let $Y$ be an ordered vector space and  $X\sub Y$. % and $y\in Y$. Then we have $y= \inf\left\{x\in X\mid| y\leq x\right\}$ if and only if there exists a set $A\sub X$ such that $y=\inf A$.
	%Moreover, 
	Then $X$ is order dense in $Y$ if and only if 
	\begin{equation*}
	\forall y\in Y \hs \exists A\sub X \colon y = \inf A.
	\end{equation*}
\end{lemma}
%%%%%%%%%%%%%%%%%%%%%%%%%%%%%%%%%%%%%%%%%%%%%%%%%%%%%%%%%
%%%%	END
%%%%%%%%%%%%%%%%%%%%%%%%%%%%%%%%%%%%%%%%%%%%%%%%%%%%%%%%%
\begin{proof}
	Let $X$ be order dense in $Y$ and let $y\in Y$.
	Then  %$y= \inf\left\{x\in X\mid| y\leq x\right\}$, hence 
	for  $A:=\left\{x\in X\mid| y\leq x\right\}$ we have $y=\inf A$.
	To see the converse implication, %first consider the set $\left\{x\in X\mid| y\leq x\right\}$ in the Dedekind completion $X^\delta$, i.e. let
	%\[S:=\left\{j(x)\in j(X)\mid| j(y)\leq j(x)\right\}.\]
	let $y\in Y$ and let $A\sub X$ with $y=\inf A$. Since for every $a\in A$ we have $y \leq a$, it follows that $A\sub B:=\left\{x\in X\mid| y\leq x\right\}$. By Lemma~\ref{properties.27.5} we obtain $y=\inf B$. Hence $X$ is order dense in $Y$.
	%The characterization of order denseness of $X$ in $Y$ follows by considering all elements $y\in Y$.
\end{proof}
%-%-%-%-%-%-%-%-%-%-%-%-%-%-%-%-%-%-%-%-%-%-%-%-%-%-%-%-
Clearly, Lemma~\ref{properties.29} can be reformulated for suprema instead of infima.

Next we define disjointness, bands and ideals in ordered vector spaces.
For a subset $M\sub X$ denote the set of upper bounds of $M$ by $M^u=\left\{x\in X\mid| \forall m\in M\colon m\leq x\right\}$. For $x\in X$, we abbreviate the set $\left\{x,-x\right\}$ by $\left\{\pm x\right\}$.
The elements $x,y\in X$ are called \emph{disjoint}, in symbols $x\perp y$, if $\left\{\pm (x+y)\right\}^u = \left\{\pm (x-y)\right\}^u$, for motivation and details see \c{1}. If $X$ is a vector lattice, then this notion of disjointness coincides with the usual one, see \c[Theorem~1.4(4)]{PosOp}. Let $Y$ be an ordered vector space, $X$ an order dense subspace of $Y$, and $x,y\in X$. Then the disjointness notions in $X$ and $Y$ coincide, i.e.\ $x\perp y$ in $X$ holds if and only if $x\perp y$ in $Y$, see \c[Proposition~2.1(ii)]{1}.
The \emph{disjoint complement} of a subset $M\sub X$ is $M^{\t{d}} := \left\{x\in X \mid| \forall y\in M\colon x\perp y\right\}$. A linear subspace $B$ of an ordered vector space $X$ is called a \emph{band} in $X$ if $B\dd = B$, see \c[Definition~5.4]{1}. If $X$ is an Archimedean vector lattice, then this notion of a band coincides with the classical notion of a band in vector lattices (where a band is defined to be an order closed ideal). For every subset $M\sub X$ the disjoint complement $M^{\t{d}}$ is a band, see \c[Proposition~5.5]{1}. In particular, we have $M^{\t{d}} = M^{\t{ddd}}$. We list further properties of bands.

%%%%%%%%%%%%%%%%%%%%%%%%%%%%%%%%%%%%%%%%%%%%%%%%%%%%%%%%%
%%%% BEGIN
%%%%%%%%%%%%%%%%%%%%%%%%%%%%%%%%%%%%%%%%%%%%%%%%%%%%%%%%%
\begin{lemma}
		Let $X$ be an ordered vector space and $\mathcal{S}$ a  non-empty set of bands in $X$. Then $\bigcap\mathcal{S}$ is a band in $X$.
\end{lemma}
%%%%%%%%%%%%%%%%%%%%%%%%%%%%%%%%%%%%%%%%%%%%%%%%%%%%%%%%%
%%%% END
%%%%%%%%%%%%%%%%%%%%%%%%%%%%%%%%%%%%%%%%%%%%%%%%%%%%%%%%%
\begin{proof}
Clearly, $(\bigcap\mathcal{S})\dd \supseteq \bigcap\mathcal{S}$. Conversely,
for every band $B\in\mathcal{S}$ we have $\bigcap\mathcal{S} \sub B$ and thus $(\bigcap\mathcal{S})\dd\sub B\dd =B$. Hence $(\bigcap\mathcal{S})\dd \sub \bigcap\mathcal{S}.$
\end{proof}

%%%%%%%%%%%%%%%%%%%%%%%%%%%%%%%%%%%%%%%%%%%%%%%%%%%%%%%%%
%%%% BEGIN
%%%%%%%%%%%%%%%%%%%%%%%%%%%%%%%%%%%%%%%%%%%%%%%%%%%%%%%%%
\begin{lemma}\label{0.0}
	Let $X$ be a pre-Riesz space, $M\sub X$ and  \begin{equation}
	\label{eq:generatedband}
	\mathcal{B}_M :=\bigcap \left\{B\sub X\mid| B \t{ is a band in } X \t{ with } M\sub B\right\}.
	\end{equation}
	Then $\mathcal{B}_M=M^{\mathrm{dd}}$.
\end{lemma}
%%%%%%%%%%%%%%%%%%%%%%%%%%%%%%%%%%%%%%%%%%%%%%%%%%%%%%%%%
%%%% END
%%%%%%%%%%%%%%%%%%%%%%%%%%%%%%%%%%%%%%%%%%%%%%%%%%%%%%%%%
\begin{proof}
As $M\dd$ is a band which contains the set $M$, we obtain $\mathcal{B}_M\subseteq M\dd$.
	
Conversely, let $x\in M\dd$. Let $B$ be a band in $X$ with $M\sub B$. Then we have  $M\dd \sub B\dd =  B$. Therefore $x\in B$. Consequently, $x\in \mathcal{B}_M$.
\end{proof}

For $M\subseteq X$ we call $\mathcal{B}_M$ as in \eqref{eq:generatedband} the \emph{band generated by} $M$.

The following notion of an ideal is introduced in \c[Definition~3.1]{vanGaa}. A subset $M$ of an ordered vector space $X$ is called \emph{solid} if for every $x\in X$ and $y\in M$ the relation
$\left\{\pm x\right\}^u \supseteq \left\{\pm y\right\}^u$ implies $x\in M$.
A solid subspace of $X$ is called an \emph{ideal}. This notion of an ideal coincides with the classical definition, provided $X$ is a vector lattice. In sharp contrast to the vector lattice setting, in an ordered vector space bands (and therefore ideals) need not be directed, see \c[Example~5.13]{2}.
%%%%%%%%%%%%%%%%%%%%%%%%%%%%%%%%%%%%%%%%%%%%%%%%%%%%%%%%%%%%%%%%%%%%%%%%
% BEGIN
%%%%%%%%%%%%%%%%%%%%%%%%%%%%%%%%%%%%%%%%%%%%%%%%%%%%%%%%%%%%%%%%%%%%%%%%
\begin{lemma}\label{prelim.30a-intersection_of_ideals}
Let $X$ be an ordered vector space and $\mathcal{S}$ a non-empty set of ideals in $X$. Then $\bigcap\mathcal{S}$ is an ideal in $X$.
\end{lemma}
%%%%%%%%%%%%%%%%%%%%%%%%%%%%%%%%%%%%%%%%%%%%%%%%%%%%%%%%%%%%%%%%%%%%%%%%
% END
%%%%%%%%%%%%%%%%%%%%%%%%%%%%%%%%%%%%%%%%%%%%%%%%%%%%%%%%%%%%%%%%%%%%%%%%
\begin{proof}
Clearly, $\bigcap\mathcal{S}$ is a linear subspace of $X$. To  show that $\bigcap\mathcal{S}$ is solid,
%By definition of $\bigcap\mathcal{S}$ we have
%\[\bigcap\mathcal{S}=\left\{x\in X\mid| \forall I\in\mathcal{S}\hs \exists a\in I\colon \left\{\pm x\right\}^u\supseteq\left\{\pm a\right\}^u\right\}.\]
 let $a\in\bigcap\mathcal{S}$ and $z\in X$ be such that $\left\{\pm z\right\}^u\supseteq\left\{\pm a\right\}^u$. For every $I\in\mathcal{S}$ we have $a\in I$ and, as $I$ is solid, also $z\in I$. We conclude $z\in \bigcap\mathcal{S}$.
\end{proof}
Let $X$ be an ordered vector space, $M\sub X$ and 
\[\mathcal{I}_M:=\bigcap \left\{I\sub X\mid| I \t{ is an ideal in } X \t{ with } M\sub I\right\}.\] We call $\mathcal{I}_M$ the \emph{ideal generated by} $M$.

By \c[Theorem~4.11]{7} a directed subspace of $X$ is an ideal if and only if it is solvex. Hereby, a set $M\sub X$ is called \emph{solvex} if for every $x\in X$, $x_1,\ldots,x_n\in M$ the inclusion
\[\left\{\pm x\right\}^u \supseteq \left\{\sum_{i=1}^n \epsilon_i \lambda_i x_i \mid| \lambda_1,\ldots,\lambda_n\in\left]0,1\right],\hs \sum_{i=1}^n\lambda_i =1,\hs \epsilon_1,\ldots,\epsilon_n\in\left\{\pm 1\right\}\right\}^u\]
implies $x\in M$, see \c[Definition~2.2]{9}. Every  solvex set is solid and convex, and the converse is true, provided $X$ is, in addition, a vector lattice, see \c[Lemma~2.3]{9}.

It is straightforward that the intersection of solvex sets is solvex. For a set $M\subseteq X$, the \emph{solvex hull} of $M$ is defined by $\t{solv}(M):=\bigcap\left\{S\sub X\mid| S \t{ solvex}, M\sub S\right\}$. 
A representation of  $\t{solv}(M)$ is given in \c[Lemma~3.4]{2}.
If $D$ is a linear subspace of $X$, then by \c[Lemma~3.5]{2} the solvex hull of $D$ is a solvex ideal. 
%with the representation 
%\begin{equation}
%\label{eq:darstellungsolvhull}
%\solv(D)=\left\{x\in X\mid|\exists x_1,\ldots,x_n\in D\colon \left\{\pm x\right\}^u\supseteq \left\{\sum_{k=1}^n\epsilon_k x_k\mid| \epsilon_k\in\left\{\pm 1\right\}\right\}^u\right\}.
%\end{equation}
For a set $M\sub X$ the \emph{solvex ideal generated by} $M$ is given by
\[\I_M^{\t{solv}}:=\bigcap\left\{I\sub X\mid| I \t{ is a solvex ideal and } M\sub I\right\}.\]

%%%%%%%%%%%%%%%%%%%%%%%%%%%%%%%%%%%%%%%%%%%%%%%%%%%%%%%%%%%%%%%%%%%%%%%%
% BEGIN
%%%%%%%%%%%%%%%%%%%%%%%%%%%%%%%%%%%%%%%%%%%%%%%%%%%%%%%%%%%%%%%%%%%%%%%%
\begin{lemma}\label{prelim.30a-solv_hull_and_solv_ideal}
Let $X$ be an ordered vector space and $M\sub X$  a non-empty subset. Then $\I_M^{\t{solv}} = \solv (\I_M)$.
\end{lemma}
%%%%%%%%%%%%%%%%%%%%%%%%%%%%%%%%%%%%%%%%%%%%%%%%%%%%%%%%%%%%%%%%%%%%%%%%
% END
%%%%%%%%%%%%%%%%%%%%%%%%%%%%%%%%%%%%%%%%%%%%%%%%%%%%%%%%%%%%%%%%%%%%%%%%
\begin{proof}
As $\I_M$ is a linear subspace of $X$, $\t{solv}(\I_M)$ is a solvex ideal.
Since $\t{solv}(\I_M)$ contains $S$ and $\I_M^{\t{solv}}$ is the smallest solvex ideal containing $S$, it follows
$\I_M^{\t{solv}} \sub \solv(\I_M) \sub \solv(\I_M^{\t{solv}}) = \I_M^{\t{solv}}$.
\end{proof}

Recall that a linear map $i\colon X\rightarrow Y$, where $X$ and $Y$ are ordered vector spaces, is called \emph{bipositive} if for every $x\in X$ one has $i(x) \geq 0$ if and only if $x\geq 0$. An \emph{embedding} is a bipositive linear map, which implies injectivity. 
For an ordered vector space $X$, the following statements are equivalent, see \c[Corollaries~4.9-11 and Theorems~3.7, 4.13]{vanHaa}:
\begin{enumerate}
	\item\label{embedding.it1} There exist a vector lattice $Y$ and an embedding $i\colon X\rightarrow Y$ such that $i(X)$ is order dense in $Y$.
	\item\label{embedding.it2} There exist a vector lattice $\tilde{Y}$ and an embedding $i\colon X\rightarrow \tilde{Y}$ such that $i(X)$ is order dense in $\tilde{Y}$ and generates $\tilde{Y}$ as a vector lattice.
\end{enumerate}
If $X$ satisfies \ref{embedding.it1}, then $X$ is called a \emph{pre-Riesz space}, and $(Y,i)$ is called a \emph{vector lattice cover} of $X$. For an intrinsic definition of pre-Riesz spaces see \c{vanHaa}. If $X$ is a subspace of $Y$ and $i$ is the inclusion map, we write briefly $Y$ for $(Y,i)$.
As all spaces $\tilde{Y}$ in \ref{embedding.it2} are Riesz isomorphic, we call the pair $(\tilde{Y},i)$ the \emph{Riesz completion of} $X$ and denote it by $X^\rho$. 
The space $X^\rho$ is the smallest vector lattice cover of $X$ in the sense that every vector lattice cover $Y$ of $X$ contains a Riesz subspace that is Riesz isomorphic to $X^\rho$.

By \c[Theorem 17.1]{vanHaa} every Archimedean directed ordered vector space is a pre-Riesz space. Moreover, every pre-Riesz space is directed.
If $X$ is an Archimedean directed ordered vector space, then every vector lattice cover of $X$ is Archimedean.  
By \c[Chapter~X.3]{VulikhWeber} every Archimedean directed  ordered vector space $X$ has a unique  Dedekind completion, which we denote by $X^\delta$. Clearly, $X^\delta$ is a vector lattice cover of $X$.

Let $X$ be a pre-Riesz space and $(Y,i)$ a vector lattice cover of $X$.
The space $X$ is called \emph{pervasive in} $Y$, if for every $y\in Y_+$, $y\neq 0$, there exists $x\in X$ such that $0<i(x) \leq y$. By \c[Proposition~2.8.8]{PRS} the space $X$ is pervasive in $Y$ if and only if $X$ is pervasive in any vector lattice cover. So, $X$ is simply called \emph{pervasive}.
The next lemma which characterizes pervasive pre-Riesz spaces was shown in the master's thesis~\c[Theorem~4.15, Corollary~4.16]{Waaij}. It was originally formulated for the more special case of integrally closed pre-Riesz spaces. A short proof for Archimedean pre-Riesz spaces can be found in \c[Lemma~1]{02}.
%%%%%%%%%%%%%%%%%%%%%%%%%%%%%%%%%%%%%%%%%%%%%%%%%%%%%%%%%
%%%% BEGIN
%%%%%%%%%%%%%%%%%%%%%%%%%%%%%%%%%%%%%%%%%%%%%%%%%%%%%%%%%
\begin{lemma}\label{properties.1}
	For an Archimedean pre-Riesz space $X$ with a vector lattice cover $(Y,i)$ the following statements are equivalent.
	\begin{enumerate}
		\item\label{properties.1.it1} $X$ is pervasive.
		\item\label{properties.1.it3} For every $y\in Y_+$ with $y\neq 0$ we have
		$y = \sup\left\{x\in i(X) \mid| 0< x\leq y\right\}$.
	\end{enumerate}
\end{lemma}
%%%%%%%%%%%%%%%%%%%%%%%%%%%%%%%%%%%%%%%%%%%%%%%%%%%%%%%%%
%%%%	END
%%%%%%%%%%%%%%%%%%%%%%%%%%%%%%%%%%%%%%%%%%%%%%%%%%%%%%%%%

A pre-Riesz space $X$ with vector lattice cover $(Y,i)$ is called \emph{fordable in} $Y$, if for every $y\in Y_+$ there exists a set $S\sub X$ such that $\left\{y\right\}^{\t{d}} = i(S)^{\t{d}}$ in $Y$. By \c[Proposition~4.1.18]{PRS} the space $X$ is fordable in $Y$ if and only if $X$ is fordable in any vector lattice cover. So, $X$ is simply called \emph{fordable}. By \c[Lemma~2.4]{3} every pervasive pre-Riesz space is fordable.
%For a mapping $i\colon A\rightarrow B$ we denote the preimage of a set $S\sub B$ by $[S]f:=\left\{x\in A\mid| f(x)\in S\right\}$.
For $S\sub Y$ we write  $[S]i:=\left\{x\in X\mid| i(x)\in S\right\}$.  The next lemma is established in \cite[Proposition~3]{02}.

%%%%%%%%%%%%%%%%%%%%%%%%%%%%%%%%%%%%%%%%%%%%%%%%%%%%%%%%%
%%%%	BEGIN
%%%%%%%%%%%%%%%%%%%%%%%%%%%%%%%%%%%%%%%%%%%%%%%%%%%%%%%%%
\begin{lemma}\label{closedness.1000}
	Let $X$ be a fordable pre-Riesz space and $(Y,i)$ a vector lattice cover of $X$. Let $S\sub X$. Then
	\[[i(S)\dd]i = S\dd.\]
\end{lemma}
%%%%%%%%%%%%%%%%%%%%%%%%%%%%%%%%%%%%%%%%%%%%%%%%%%%%%%%%%
%%%%	END
%%%%%%%%%%%%%%%%%%%%%%%%%%%%%%%%%%%%%%%%%%%%%%%%%%%%%%%%%

Next we discuss the restriction property and the extension property for ideals and bands.
Let $X$ be a pre-Riesz space and $(Y,i)$ a vector lattice cover of $X$. 
The pair $(L,M)\subseteq \mathcal{P}(X)\times \mathcal{P}(Y)$
is said to satisfy  
\begin{itemize}
	\item[-]
	the \emph{restriction property} (R), if whenever $J\in M$,
	then $[J]i\in L$, and
	%$J\subseteq Y$ has property (P) in $Y$, then the set $[J]i$ %$J\cap i(X)$ 
	%has property (P) in $X$,
	\item[-]
	the \emph{extension property} (E), if whenever $I\in L$,
	%$I\subseteq X$ has property (P) in $X$, 
	then there is $J\in M$ such that 
	%$J\subseteq Y$ such that $J$ has property (P) in $Y$ and
	$I=[J]i$.
	%$i(I)=J\cap i(X)$.  
\end{itemize}
%In \c{2} the following \emph{restriction property} (R) and \emph{extension property} (E) for a property $P$ are considered:
%\begin{tabbing}
%\quad\=(R) \= If $J\sub Y$ has property $P$ in $Y$, then $[J\cap i(X)]i$ has property $P$ in $X$.\\
%\quad\=(E) \= If $I\sub X$ has property $P$ in $X$, then there exists a subset $J\sub Y$ with\\ 
%      \>\> property $P$ such that $i(I) = J\cap i(X)$ in $Y$.
%\end{tabbing}
In \c{1} the properties (R) and (E) are investigated for bands, solvex ideals and ideals.
More precisely, it is shown 
%By \c[Propositions~5.12, 5.3 and 5.1(iii)]{1}, 
that the extension property (E) is satisfied for  bands, i.e.\ for the pair $(L,M)$ where $L$ is the set of all bands in $X$ and $M$ is the set of all bands in $Y$.
Moreover, the restriction property (R) is satisfied for ideals. In general, bands do not have (R) and ideals do not have (E). %For more details, see the overview table in \c[p.~603]{1}. 
%By \c[Propositions~5.5(i) and 5.6]{1}, 
The pair $(L,M)$, where $L$ and $M$ are the sets of all solvex ideals in $X$ and $Y$, respectively, satisfies both (E) and (R). If $X$ is fordable, then by \c[Proposition~2.5 and Theorem~2.6]{3} we have (R) for bands.

%Property $P$ might be the property of being an ideal or a band. 
If for an ideal $I$ in $X$ and an ideal $J$ in $Y$ we have $I=[J]i$, then $J$ is called an \emph{extension ideal} of $I$. An \emph{extension band} $J$ for a band $I$ in $X$ is defined similarly. Extension ideals and bands are not unique, in general.
 %Recall that the intersection of ideals is an ideal, and that the similar statement holds for bands. Therefore, 
If an ideal $I$ in $X$ has an extension ideal in $Y$, then
\[\hat{I}:=\bigcap\left\{J\subseteq Y \mid| J \t{ is an extension ideal of } I\right\}\]
is the smallest extension ideal of $I$ in $Y$. Similarly, for every band $B$ in $X$ the smallest extension band $\hat{B}$ is defined.
For \ref{properties.-1.it1} in the next proposition see \c[Proposition 5.6]{2}, for \ref{properties.-1.it2} see \c[Proposition 17 (a)]{6} and its subsequent discussion.
%%%%%%%%%%%%%%%%%%%%%%%%%%%%%%%%%%%%%%%%%%%%%%%%%%%%%%%%%
%%%% BEGIN
%%%%%%%%%%%%%%%%%%%%%%%%%%%%%%%%%%%%%%%%%%%%%%%%%%%%%%%%%
\begin{proposition}\label{properties.-1}
Let $X$ be a pre-Riesz space and $(Y,i)$ a vector lattice cover of $X$.
\begin{enumerate}
\item\label{properties.-1.it1} For a solvex ideal $I$ in $X$ we have $\hat{I}=\I_{i(I)}$ in $Y$.
\item\label{properties.-1.it2} For a band $B$ in $X$ we have $\hat{B}=\B_{i(B)}$ in $Y$.
\end{enumerate}
\end{proposition}
%%%%%%%%%%%%%%%%%%%%%%%%%%%%%%%%%%%%%%%%%%%%%%%%%%%%%%%%%
%%%% END
%%%%%%%%%%%%%%%%%%%%%%%%%%%%%%%%%%%%%%%%%%%%%%%%%%%%%%%%%

The following well-known property of ideals can be found in \c[Theorem~IV.1.2]{Vulikh_en}.
%%%%%%%%%%%%%%%%%%%%%%%%%%%%%%%%%%%%%%%%%%%%%%%%%%%%%%%%%
%%%% BEGIN
%%%%%%%%%%%%%%%%%%%%%%%%%%%%%%%%%%%%%%%%%%%%%%%%%%%%%%%%%
\begin{lemma}\label{properties.0}
Let $X$ be a Dedekind complete vector lattice and $I\sub X$ an ideal. Then $I$ is a Dedekind complete sublattice of $X$.
\end{lemma}
\section{Representation of directed ideals generated by a set}

In this section $X$ is a pre-Riesz space. 
We intend to give an explicit formula for the ideal generated by a set of positive elements.
This formula is similar to the vector lattice case.
%%%%%%%%%%%%%%%%%%%%%%%%%%%%%%%%%%%%%%%%%%%%%%%%%%%%%%%%%%%%%%%%%%%%%%%%
% BEGIN
%%%%%%%%%%%%%%%%%%%%%%%%%%%%%%%%%%%%%%%%%%%%%%%%%%%%%%%%%%%%%%%%%%%%%%%%
\begin{proposition}\label{prelim.30a}
Let $X$ be a pre-Riesz space and $S\sub X_+$ a non-empty subset. Then the ideal $\I_S$ generated by $S$ is directed and
\[\I_S=\left\{x\in X\mid| \exists x_1,\ldots,x_n \in S \hs\exists\lambda\in\RR_{\geq 0} \colon \pm x \leq\lambda \sum_{i=1}^n x_i\right\}.\]
\end{proposition}
%%%%%%%%%%%%%%%%%%%%%%%%%%%%%%%%%%%%%%%%%%%%%%%%%%%%%%%%%%%%%%%%%%%%%%%%
% END
%%%%%%%%%%%%%%%%%%%%%%%%%%%%%%%%%%%%%%%%%%%%%%%%%%%%%%%%%%%%%%%%%%%%%%%%
\begin{proof}
Let
\[A: = \left\{x\in X\mid| \exists x_1,\ldots,x_n \in S \hs\exists\lambda\in\RR_{\geq 0} \colon \pm x \leq\lambda \sum_{i=1}^n x_i\right\}.\]
It is easy to see that $S\sub A$ holds and the set $A$ is a linear subspace. To see that $A$ is solid, let $x\in X$ be such that there exists an $a\in A$ with $\left\{\pm x\right\}^u \supseteq \left\{\pm a\right\}^u$.  Since $a\in A$, there exist $x_1,\ldots, x_n\in S$ and a $\lambda\in\RR_{\geq 0}$ with $\pm a \leq\lambda \sum_{i=1}^n x_i$, i.e.\ $z:=\lambda \sum_{i=1}^n x_i \in \left\{\pm a\right\}^u$. It follows $z\in\left\{\pm x\right\}^u$, i.e.\ $\pm x\leq \lambda \sum_{i=1}^n x_i$. That is, $x\in A$. We conclude that $A$ is solid.

Next we show that $A$ is directed. Let $x\in A$. Then there are  $x_1,\ldots, x_n\in S$ and a $\lambda\in\RR_+$ with $\pm x \leq\lambda\sum_{i=1}^n x_i =:z$. 
Clearly, $z\in A_+$. Thus $x=z-(z-x)$ is a decomposition of the element $x$ into two positive elements of $A$, i.e.\ $A$ is directed.

We show $A=\I_S$. For the inclusion $A\sub \I_S$, let $x\in A$ and $z:=\lambda\sum_{i=1}^n x_i$ be as above. Notice that $z\in \I_S$. As $\left\{\pm x\right\}^u\supseteq \left\{z\right\}^u=\left\{\pm z\right\}^u$, due to $\I_S$ being solid it follows $x\in\I_S$. The converse inclusion $\I_S\sub A$ is immediate, since $A$ is an ideal containing the set $S$. Altogether, we have $A=\I_S$.
\end{proof}
\begin{corollary}\label{corsomething}
Let $X$ be a pre-Riesz space and $S$ a directed linear subspace of $X$. Then the ideal $\I_S$ is directed.
\end{corollary}
%%%%%%%%%%%%%%%%%%%%%%%%%%%%%%%%%%%%%%%%%%%%%%%%%%%%%%%%%
%%%% END
%%%%%%%%%%%%%%%%%%%%%%%%%%%%%%%%%%%%%%%%%%%%%%%%%%%%%%%%%
\begin{proof}
Due to $S=S_+-S_+$ and $\I_{S_+}=\I_S$ by Proposition~\ref{prelim.30a} we obtain that $\I_S$ is directed.
\end{proof}

Corollary~\ref{corsomething} leads to the question whether an analogous statement is true for a band $B$ in $X$. That is, if $B$ is generated by a directed linear subspace, is then $B$ likewise directed?
The following example shows that the answer is negative even if $X$ is pervasive, has the RDP and an order unit.
%%%%%%%%%%%%%%%%%%%%%%%%%%%%%%%%%%%%%%%%%%%%%%%%%%%%%%%%%
%%%% BEGIN
%%%%%%%%%%%%%%%%%%%%%%%%%%%%%%%%%%%%%%%%%%%%%%%%%%%%%%%%%
\begin{example}\label{Namioka}
\textit{In a pervasive pre-Riesz space with RDP and an order unit a directed ideal might generate a non-directed band.}

The ordered vector space 
\[X:=\left\{f\in C[-1,1] \mid| f(0) = \frac{f(-1)+f(1)}{2}\right\}\]
endowed with pointwise order
is first considered by Namioka in \c[Example~8.10]{Namioka1957}, where it is shown that $X$ is not a vector lattice, but has the RDP (see also \c[V.2, Beispiel 9]{VulikhWeber}). From \c[Example 18]{6} it follows that $X$ is an Archimedean pre-Riesz space and $C[0,1]$ is a vector lattice cover of $X$. In \c[Example~12]{02} it is established that $X$ is pervasive.
%The space $X$ is Archimedean as a subspace of $C[-1,1]$ and directed, since it contains the order unit $\1_{[-1,1]}$. Therefore $X$ is pre-Riesz. Moreover, it is shown in \c[V.2, Example 9]{Vulikh_introduction} that $X$ is not a vector lattice, but has the RDP. By \c[Example 18]{6} the space $C[-1,1]$ is a vector lattice cover of $X$. Furthermore, $X$ is pervasive. Indeed, let $y\in C[-1,1]$ be with $y> 0$. There exists a $p\in[-1,1]$ and a neighbourhood $U(p)$ of $p$ such that $y(U(p))>0$. Then we can find an $x\in C[-1,1]$ such that $x(-1)=x(0)=x(1)=0$ and $0<x(t)\leq y(t)$ for all $t$ in a non-empty subset of $U(p)\ohne\left\{-1,0,1\right\}$. We have $0<x\leq y$ and $x\in X$, i.e.\ $X$ is pervasive.
Moreover, it is shown that the ideal
\[I:=\left\{f\in X \mid| f([-\textstyle{\frac{1}{2}},0]\cup\left\{-1,1\right\}) = 0 \right\}\]
in $X$ is not a band. Clearly, $I$ is directed. 
For the band generated by $I$ we have $\B_I=I\dd=\left\{f\in X \mid| f([-\textstyle{\frac{1}{2}},0]) = 0\right\}$, which is not directed by \c[Example~12]{02}.
\end{example}
%%%%%%%%%%%%%%%%%%%%%%%%%%%%%%%%%%%%%%%%%%%%%%%%%%%%%%%%%
%%%% END
%%%%%%%%%%%%%%%%%%%%%%%%%%%%%%%%%%%%%%%%%%%%%%%%%%%%%%%%%

%%%%%%%%%%%%%%%%%%%%%%%%%%%%%%%%%%%%%%%%%%%%%%%%%%%%%%%%%
%%%%
%%%%
%%%%		Extension of generated ideals and bands
%%%%
%%%%
%%%%%%%%%%%%%%%%%%%%%%%%%%%%%%%%%%%%%%%%%%%%%%%%%%%%%%%%%
\section{Extension of ideals and bands generated by sets}

In this section $X$, is a pre-Riesz space and $(Y,i)$ a vector lattice cover of $X$. For a set $S\sub X$ we intend to relate $\I_S$ in $X$ to $\I_{i(S)}$ in $Y$, and, similarly, $\B_S$ in $X$ to $\B_{i(S)}$ in $Y$. These results will be used in Section~\ref{extensions}.

%%%%%%%%%%%%%%%%%%%%%%%%%%%%%%%%%%%%%%%%%%%%%%%%%%%%%%%%%%%%%%%%%%%%%%%%
% BEGIN
%%%%%%%%%%%%%%%%%%%%%%%%%%%%%%%%%%%%%%%%%%%%%%%%%%%%%%%%%%%%%%%%%%%%%%%%
\begin{proposition}\label{inclusion}
Let $X$ be a pre-Riesz space and $(Y,i)$ a vector lattice cover of $X$. Let $S\sub X$ be a non-empty subset. Then $i(\I_S)\sub\I_{i(S)}\cap i(X)$.
\end{proposition}
%%%%%%%%%%%%%%%%%%%%%%%%%%%%%%%%%%%%%%%%%%%%%%%%%%%%%%%%%%%%%%%%%%%%%%%%
% END
%%%%%%%%%%%%%%%%%%%%%%%%%%%%%%%%%%%%%%%%%%%%%%%%%%%%%%%%%%%%%%%%%%%%%%%%
\begin{proof}
Let $x\in\I_S$. It is sufficient to show $i(x)\in\I_{i(S)}$. Define the two sets 
\[\mathcal{R}:=\left\{I\sub X\mid| I \t{ ideal and }S\sub I\right\} \hs\hs\t{and}\hs\hs \mathcal{S}:=\left\{I\sub X\mid| I \t{ solvex ideal and }S\sub I\right\}.\] By the definition of $\I_S$ we have $x\in\bigcap \mathcal{R}$.
%Notice first that due to (R) for ideals for every ideal $J\sub Y$ with $i(S)\sub J$ there exists an ideal $I\sub X$ with $S\sub I$ and $i(I)\sub J$ (e.g.\ the restriction ideal $I:=[J]i$). This implies the inclusion relation in
As an ideal is solvex if and only if it has an extension ideal in $Y$, we obtain
\[i(x)\in i\left(\bigcap_{I\in\mathcal{R}}I\right) \hs\hs=\hs\hs \bigcap_{I\in\mathcal{R}}i(I) \hs\hs \sub\hs\hs  \bigcap_{I\in\mathcal{S}}i(I)\hs\hs\sub \bigcap_{\substack{J\t{ ideal in } Y,\\ i(S)\sub J}} J  \hs\hs=\hs\hs \I_{i(S)}.\]
\end{proof}

If in the setting of Proposition~\ref{inclusion} we have $S\sub X_+$, then the subsequent result establishes, in particular, the equality $i(\I_S)=\I_{i(S)}\cap i(X)$.
\begin{proposition}\label{genideal=erwideal}
Let $X$ be a pre-Riesz space and $(Y,i)$ a vector lattice cover of $X$. Let $S\sub X_+$ be a non-empty subset. Then the smallest extension ideal of the ideal $\I_S$ is the ideal $\I_{i(S)}$ in $Y$.
\end{proposition}
%%%%%%%%%%%%%%%%%%%%%%%%%%%%%%%%%%%%%%%%%%%%%%%%%%%%%%%%%%%%%%%%%%%%%%%%
% END
%%%%%%%%%%%%%%%%%%%%%%%%%%%%%%%%%%%%%%%%%%%%%%%%%%%%%%%%%%%%%%%%%%%%%%%%
\begin{proof}
We first show that $\I_{i(S)}$ is an extension ideal of $\I_S$. 
%Since ideals have the restriction property, it follows that $[\I_{i(S)}]i\sub X$ is an ideal. Moreover, the ideal $[\I_{i(S)}]i$ contains the set $S\sub X$. It follows $\I_S\sub [\I_{i(S)}]i$. 
The inclusion $\I_S\sub [\I_{i(S)}]i$ follows from Proposition~\ref{inclusion}.
For the converse inclusion recall that the ideal $\I_{i(S)}$ in the vector lattice $Y$ has the form
\[\I_{i(S)} = \left\{y\in Y\mid| \exists s_1,\ldots,s_n\in S\hs \exists \lambda\geq 0\colon |y|\leq\lambda\sum_{k=1}^n i(s_k)\right\}.\]
Thus for every $x\in X$ with $i(x)\in \I_{i(S)}$ there exist $s_1,\ldots, s_n \in S$ such that $\pm x \leq \lambda \sum_{i=1}^n s_i$. By Proposition~\ref{prelim.30a} it follows $x\in \I_S$. Altogether we have $\I_S = [\I_{i(S)}]i$, i.e.\ $\I_{i(S)}$ is an extension ideal of $\I_S$.
It is left to show that $\I_{i(S)}$ is the smallest extension ideal of $\I_S$. This follows immediately, as every extension ideal of $\I_S$ has to contain the set $i(S)$.
%Let $S\sub X$ be non-empty such that $\I_S$ is a solvex ideal in $X$. We first show that $\I_{i(S)}$ is an extension ideal of $\I_S$. We need to establish $i(\I_S)=\I_{i(S)}\cap i(X)$. 
%
%Observe that the inclusion $i(\I_S)\sub\I_{i(S)}\cap i(X)$ follows by Proposition~\ref{inclusion}.
%For the converse inclusion, let $x\in X$ be such that $i(x)\in \I_{i(S)}\cap i(X)$. Since $\I_{i(S)}$ is an ideal in the vector lattice $Y$, it has the form
%\[\I_{i(S)} = \left\{y\in Y\mid| \exists s_1,\ldots,s_n\in S\hs \exists \lambda\geq 0\colon |y|\leq\lambda\sum_{k=1}^n|i(s_k)|\right\}.\]
%Thus there exist $s_1,\ldots,s_n\in S$ and a $\lambda\geq 0$ such that $\pm i(x)\leq\lambda\sum_{k=1}^n|i(s_k)|$. In particular, we have for every $u_k\in\left\{\pm s_k\right\}^u$ that $\pm x \leq \lambda \sum_{k=1}^n u_k$. Theorem~\ref{prelim.30a-hull} yields $x\in\I^{\t{solv}}_S$. Since the ideal $\I_S$ is solvex, we have $\I^{\t{solv}}_S=\I_S$. Consequently, $x\in\I_S$. This establishes the inclusion $\I_{i(S)}\cap i(X) \sub i(\I_S)$.
%
%To sum up, we have $\I_{i(S)}\cap i(X) = i(\I_S)$, that is, $\I_{i(S)}$ is an extension ideal of $\I_S$. It is left to show that $\I_{i(S)}$ is the smallest extension ideal of $\I_S$. This follows immediately, as every extension ideal of $\I_S$ has to contain the set $i(S)$.
\end{proof}

The following result gives conditions, under which an ideal in a pre-Riesz space is directed. In particular, it shows that directed ideals are precisely those ideals which have an extension ideal generated by positive elements of the pre-Riesz space.
%%%%%%%%%%%%%%%%%%%%%%%%%%%%%%%%%%%%%%%%%%%%%%%%%%%%%%%%%
%%%% BEGIN
%%%%%%%%%%%%%%%%%%%%%%%%%%%%%%%%%%%%%%%%%%%%%%%%%%%%%%%%%
\begin{theorem}\label{something}
Let $X$ be a pre-Riesz space with a vector lattice cover $(Y,i)$ and let $I\sub X$ be an ideal. Then the following statements are equivalent.
\begin{enumerate}
\item\label{something.i1} $I$ is directed.
\item\label{something.i3} There exists a set $S\sub X_+$ of positive elements such that $I=\I_S$.
\item\label{something.i2} The set $i(I)$ is majorizing in $\I_{i(I)}$.
\item\label{something.i4} $I$ has an extension ideal and for the smallest extension ideal $\hat{I}$ of $I$ there exists a set $S\sub X_+$ such that $\hat{I}=\I_{i(S)}$.
\end{enumerate}
If one of the previous statements is true, then in \ref{something.i3} and \ref{something.i4} we can choose $S:=I_+$.
\end{theorem}
%%%%%%%%%%%%%%%%%%%%%%%%%%%%%%%%%%%%%%%%%%%%%%%%%%%%%%%%%
%%%% END
%%%%%%%%%%%%%%%%%%%%%%%%%%%%%%%%%%%%%%%%%%%%%%%%%%%%%%%%%
\begin{proof}
\ref{something.i1} $\Rightarrow$ \ref{something.i3}: Let $S:=I_+$.
Clearly, since $I$ is directed and $\I_S$ is a vector space containing $S$, we have $I=I_+-I_+\sub \I_S$.
On the other hand, due to $S\sub I$ it follows $\I_S\sub \I_I=I$.

\ref{something.i3} $\Rightarrow$ \ref{something.i1}: This implication follows from Proposition~\ref{prelim.30a}.

\ref{something.i1} $\Rightarrow$ \ref{something.i2}: Let $I$ be a directed ideal in $X$ and let $z\in\I_{i(I)}$. In the vector lattice $Y$ for the ideal $\I_{i(I)}$ we have
\[\I_{i(I)} =\left\{y\in Y \mid| \exists a\in i(I)\colon |y|\leq |a|\right\}.\]
That is, there exists an $a\in i(I)$ such that $|z|\leq |a|$. Since $I$ is directed, there are $a_1,a_2\in i(I_+)=i(I)_+$ with $a=a_1-a_2$. We obtain the inequality $z\leq |z|\leq |a| = a^++a^- \leq a_1+a_2\in i(I_+)$. Therefore $i(I)$ is majorizing in $\I_{i(I)}$.

\ref{something.i2} $\Rightarrow$ \ref{something.i1}: Let $i(I)$ be majorizing in $\I_{i(I)}$ and let $x\in I$. Then  $|i(x)|\in \I_{i(I)}$ is majorized by a positive element $i(a)\in i(I)$, i.e.\ $i(x)\leq |i(x)| \leq i(a)$ and thus $x\leq a$. Let $x_1:=a\geq 0$ and $x_2:=a-x\geq 0$. Then $x=x_1-x_2$ for $x_1,x_2\in I_+$. That is, $I_+$ is generating and therefore $I$ is directed.

\ref{something.i3} $\Rightarrow$ \ref{something.i4}: This implication follows from Proposition~\ref{genideal=erwideal}. 

\ref{something.i4} $\Rightarrow$ \ref{something.i3}: Let $I$ have the smallest extension ideal $\hat{I}$ in $Y$ such that $\hat{I}=\I_{i(S)}$ for a set $S\sub X_+$. By Proposition~\ref{genideal=erwideal}
we have $\I_S=[\I_{i(S)}]i= I$.
\end{proof}

Note that, contrary to the case of ideals, a set $S\sub X_+$ may generate a non-directed band $\B_S$, see Example~\ref{Namioka} with $S:=I_+$.

%%%%%%%%%%%%%%%%%%%%%%%%%%%%%%%%%%%%%%%%%%%%%%%%%%%%%%%%%
%%%% BEGIN
%%%%%%%%%%%%%%%%%%%%%%%%%%%%%%%%%%%%%%%%%%%%%%%%%%%%%%%%%
\begin{corollary}\label{corsomething1}
Let $X$ be a pre-Riesz space and $(Y,i)$ a vector lattice cover of $X$. Let $S$ be a non-empty subset of $X_+$. Then we have $\I_{i(S)}=\I_{i(\I_S)}$.
\end{corollary}
%%%%%%%%%%%%%%%%%%%%%%%%%%%%%%%%%%%%%%%%%%%%%%%%%%%%%%%%%
%%%% END
%%%%%%%%%%%%%%%%%%%%%%%%%%%%%%%%%%%%%%%%%%%%%%%%%%%%%%%%%
\begin{proof}
Clearly, $i(S)\sub i(\I_S)$ and therefore $\I_{i(S)}\sub\I_{i(\I_S)}$. For the converse inclusion, notice that by Proposition~\ref{genideal=erwideal} the smallest extension ideal of $\I_S$ in $Y$ is the ideal $\I_{i(S)}$. Hence $i(\I_S)\sub \I_{i(S)}$, which implies $\I_{i(\I_S)}\sub \I_{i(S)}$.
\end{proof}

Next we investigate under which conditions an analogue of Proposition~\ref{genideal=erwideal} is true for bands. 
By Proposition~\ref{properties.-1} for a band $B\sub X$ the smallest extension band is given by $\B_{i(B)}$. Let $B=\B_S$ for a non-empty set $S\sub X$. Under which conditions is the smallest extension band of $B$ given by $\B_{i(S)}$?
\begin{theorem}\label{something.a}
Let $X$ be a fordable pre-Riesz space and $(Y,i)$ a vector lattice cover of $X$. Let $S\sub X$ be a non-empty subset. Then the smallest extension band of $\B_S$ is the band $\B_{i(S)}$.
\end{theorem}
%%%%%%%%%%%%%%%%%%%%%%%%%%%%%%%%%%%%%%%%%%%%%%%%%%%%%%%%%
%%%%	END
%%%%%%%%%%%%%%%%%%%%%%%%%%%%%%%%%%%%%%%%%%%%%%%%%%%%%%%%%
\begin{proof}
First notice that by Lemma~\ref{closedness.1000}
and Lemma~\ref{0.0}
 we have
\[ [\B_{i(S)}]i=[i(S)\dd]i=S\dd=\B_S,\]
that is, the band $\B_{i(S)}$ is an extension band of $\B_S$.
Due to $S\sub \B_S$ every extension band of $\B_S$ contains the set $i(S)$. The band $\B_{i(S)}$ is the smallest band in $Y$ containing $i(S)$ and thus is the smallest extension band of $\B_S$.
\end{proof}

In the following example we show that in Theorem~\ref{something.a} the condition of $X$ being fordable can not be omitted.
%%%%%%%%%%%%%%%%%%%%%%%%%%%%%%%%%%%%%%%%%%%%%%%%%%%%%%%%%
%%%% BEGIN
%%%%%%%%%%%%%%%%%%%%%%%%%%%%%%%%%%%%%%%%%%%%%%%%%%%%%%%%%
\begin{example}\label{fordabilit_necessary}
\textit{A non-fordable pre-Riesz space in which for a subset $S$ the band $\B_{i(S)}$ is not the extension band of $\B_S$.}

Let $X:=\RR^3$. Consider in $X$ the four vectors
\[v_1=\smallmath[.7]{\left(\begin{array}{c} 1\\ 0\\ 1 \end{array}\right)},\hs
	v_2=\smallmath[.7]{\left(\begin{array}{c} 0\\ 1\\ 1 \end{array}\right)},\hs
	v_3=\smallmath[.7]{\left(\begin{array}{c} -1\\ 0\\ 1 \end{array}\right)},\hs
	v_4=\smallmath[.7]{\left(\begin{array}{c} 0\\ -1\\ 1 \end{array}\right)},
\]
and let the order on $X$ be induced by the \textit{four ray cone} $K_4:=\t{pos}\left\{v_1,v_2,v_3,v_4\right\}$, i.e.\ by the positive-linear hull of the vectors $v_1,\ldots,v_4$.
By \c[Proposition~13]{5} the pre-Riesz space $(X,K_4)$ can be order densely embedded into $\RR^4$ equipped with the standard cone $\RR^4_+$. Recall that $X'$ is identified with $\RR^3$. By \c[Example~4.8]{2} the cone $K_4$ can be represented using the functionals
\[f_1=\smallmath[.7]{\left(\begin{array}{c} -1\\ -1\\ 1 \end{array}\right)},\hs 
	f_2=\smallmath[.7]{\left(\begin{array}{c} 1\\ -1\\ 1 \end{array}\right)},\hs
	f_3=\smallmath[.7]{\left(\begin{array}{c} 1\\ 1\\ 1 \end{array}\right)},\hs
	f_4=\smallmath[.7]{\left(\begin{array}{c} -1\\ 1\\ 1 \end{array}\right)}
\]
in $X'$, namely
$K_4=\left\{x\in \RR^n \mid| \forall i\in\left\{1,\ldots,4\right\}\colon f_i(x) \geq 0\right\}$.
The map $i\colon X\rightarrow \RR^4$, given by $i\colon x\mapsto (f_1(x),\ldots,f_4(x))^T$,
is a bipositive embedding of $(X, K_4)$ into the vector lattice cover $(\RR^4,\RR^4_+)$. %Notice that by \c[Proposition~13]{5} the space $(\RR^4,\RR^4_+)$ is the Riesz completion of $(X, K_4)$. Moreover, since $(\RR^4,\RR^4_+)$ is Dedekind complete, it is at the same time the Dedekind completion of $(X, K_4)$.
%By \c[Example~4.8]{2} for the image $i(X)\sub \RR^4$ we have
%%\begin{align}\label{atomic.10.9.eq1}
%i(X) =\lin\left\{b_1=\smallmath[.7]{\left(\begin{array}{c} 1\\ 1\\ 0\\ 0 \end{array}\right)}, 
%	b_2=\smallmath[.7]{\left(\begin{array}{c} 0\\ 1\\ 1\\ 0 \end{array}\right)}, 
%	b_3=\smallmath[.7]{\left(\begin{array}{c} 0\\ 0\\ 1\\ 1 \end{array}\right)}\right\}
%	=\left\{\smallmath[.7]{\left(\begin{array}{c} \lambda_1 \\ \lambda_1+\lambda_2 \\ \lambda_2+\lambda_3 \\ \lambda_3 \end{array}\right)} \mid| \lambda_1,\lambda_2,\lambda_3\in\RR\right\}.
%\end{align}

Let $S:=\left\{v_1,v_4\right\}$. Observe that $\t{span}\left\{v_1,v_4\right\}=\t{ker}f_4$. In \c[Example~5.13]{2} it is established that all non-trivial bands in $X$ are one-dimensional. Therefore the smallest band in $X$ containing $S$ is the whole space $X$, i.e.\ we have $\B_S=X$. Since $i(S)=\left\{(0,1,1,0)^T,(1,1,0,0)^T\right\}$, it follows $(i(S))^{\t{d}}=\t{span}\left\{(0,0,0,1)^T\right\}$. This yields
\[\B_{i(S)} = \t{span}\left\{(1,0,0,0)^T,(0,1,0,0)^T,(0,0,1,0)^T\right\}.\]
In \c[Example~5.13]{2} it is shown $[\B_{i(S)}]i =\t{span}\left\{v_1,v_4\right\} \neq\B_S$. It follows that $\B_{i(S)}$ is not an extension band of $\B_S$.

As was already mentioned in \c[Example~5.13]{2}, the pre-Riesz space $X$ does not have the restriction property (R) for bands, since the restriction $\t{span}\left\{v_1,v_4\right\}$ of the band $\B_{i(S)}$ to $X$ is not a band. As every fordable pre-Riesz space has (R) for bands, we conclude that $X$ is not fordable.
\end{example}

In a pre-Riesz space, by Theorem~\ref{something}  an ideal $I$ is directed if and only if $i(I)$ is majorizing in $\I_{i(I)}$.  Is the same equivalence true for every band $B$ in $X$? Clearly, if $i(B)$ is majorizing in $\B_{i(B)}$, then $B$ is directed. The next example demonstrates that the converse is not true even under additional conditions on the space $X$.
%%%%%%%%%%%%%%%%%%%%%%%%%%%%%%%%%%%%%%%%%%%%%%%%%%%%%%%%%
%%%% BEGIN
%%%%%%%%%%%%%%%%%%%%%%%%%%%%%%%%%%%%%%%%%%%%%%%%%%%%%%%%%
\begin{example}\label{ex2}
\textit{A pervasive pre-Riesz space with RDP and an order unit, where for a directed band $B$ the linear subspace $i(B)$ is not majorizing in the smallest extension band $\hat{B}$ of $B$.}

Let $X=C^1[0,1]$ be the vector space of continuously differentiable functions on the interval $[0,1]$ endowed with pointwise order. From \c[Example~3.7]{3} it follows that $Y:=C[0,1]$ is a vector lattice cover of the pre-Riesz space $X$ and that $X$ is pervasive. Let
\[S:= \left\{f\in X_+ \mid| f([0,\tfrac{1}{2}]\cup\left\{1\right\}) =0\right\}.\]
Recall that for $x,y\in X$ we have $x\perp y$ in $X$ if and only if $i(x)\perp i(y)$ in $Y$. Thus the band generated by $S$ has the form
$B:=\B_S=\left\{f\in X \mid| f([0,\tfrac{1}{2}]) =0\right\}$.
Clearly, $\B_S$ is directed.
For the band $\B_{i(S)}$ in $Y$ we have
\[\B_{i(S)} = \left\{f\in Y \mid| f([0,\tfrac{1}{2}])=0\right\}.\]
As pervasiveness implies fordability, we can apply Theorem~\ref{something.a}. It yields that $\B_{i(S)}$ is the smallest extension band of $B$. That is, we have $\B_{i(S)}=\hat{B}=\B_{i(B)}$. However, $i(B)$ is not majorizing in $\B_{i(S)}$. Indeed, let $g\in Y$ be defined by $g(t)=0$ for $t\in[0,\tfrac{1}{2}]$ and $g(t)=t-\tfrac{1}{2}$ for $t\in\hs]\tfrac{1}{2},1]$. Then we have $g\in\B_{i(S)}$, but there does not exist a continuously differentiable function $f\in B$ such that $g\leq f$.
\end{example}
%%%%%%%%%%%%%%%%%%%%%%%%%%%%%%%%%%%%%%%%%%%%%%%%%%%%%%%%%
%%%% END
%%%%%%%%%%%%%%%%%%%%%%%%%%%%%%%%%%%%%%%%%%%%%%%%%%%%%%%%%

In the remainder of this section we provide an application of Theorem~\ref{something}. The result in Theorem~\ref{closedness.7supperp} below is a generalization of a well-known statement for Archimedean vector lattices. First we recall a technical result for Archimedean pervasive pre-Riesz spaces, which is given in \c[Theorem~10]{02}.

%%%%%%%%%%%%%%%%%%%%%%%%%%%%%%%%%%%%%%%%%%%%%%%%%%%%%%%%%
%%%%				1.3.5 Theorem
%%%%%%%%%%%%%%%%%%%%%%%%%%%%%%%%%%%%%%%%%%%%%%%%%%%%%%%%%
\begin{proposition}\label{1.3.5}
Let $X$ be an Archimedean pervasive pre-Riesz space and let $I$ be a directed ideal in $X$. 
Then
\begin{equation*}
\Big\{x\in X_+ \hs\Big|\hs x=\sup\left\{a\in I\mid| 0\leq a\leq x\right\}\Big\} = (I\dd)_+.
\end{equation*}
\end{proposition}
%%%%%%%%%%%%%%%%%%%%%%%%%%%%%%%%%%%%%%%%%%%%%%%%%%%%%%%%%
%%%% END
%%%%%%%%%%%%%%%%%%%%%%%%%%%%%%%%%%%%%%%%%%%%%%%%%%%%%%%%%

%%%%%%%%%%%%%%%%%%%%%%%%%%%%%%%%%%%%%%%%%%%%%%%%%%%%%%%%%
%%%% BEGIN
%%%%%%%%%%%%%%%%%%%%%%%%%%%%%%%%%%%%%%%%%%%%%%%%%%%%%%%%%
\begin{theorem}\label{closedness.7supperp}
Let $X$ be an Archimedean pervasive pre-Riesz space, $a\in X$ and $S\sub X_+$ such that $\sup S$ exists in $X$. Then the relation $a\perp S$ implies $a\perp\sup S$.
\end{theorem}
%%%%%%%%%%%%%%%%%%%%%%%%%%%%%%%%%%%%%%%%%%%%%%%%%%%%%%%%%
%%%% END
%%%%%%%%%%%%%%%%%%%%%%%%%%%%%%%%%%%%%%%%%%%%%%%%%%%%%%%%%
\begin{proof}
Let $S\sub X_+$ be such that $\sup S$ exists in $X$ and let $a\in X$ with $a\perp S$. Consider the set
$M:=\left\{x\in X\mid| \exists s\in S\colon 0\leq x\leq s\right\}$. Clearly, $M=\left\{x\in \I_S\mid| \exists s\in S\colon 0\leq x\leq s\right\}$.
Since $S\sub M$, by Lemma~\ref{properties.27.5} we obtain that $\sup M$ exists in $X$ and $\sup M=\sup S$.
%From $\B_S=S\dd\sub\I_S\dd$ and $\I_S\dd \sub \B_S\dd=\B_S$ it follows $\I_S\dd=\B_S$.
Let $A:=\left\{x\in\I_S\mid|0\leq x \leq \sup M\right\}$. Due to $M\sub A$ by Lemma~\ref{properties.27.5} we obtain that $\sup A$ exists in $X$ and equals $\sup M$. To sum up, we have $\sup M=\sup\left\{x\in\I_S\mid| 0\leq x\leq \sup M\right\}$.

We show $a\perp \sup M$. Since the set $S$ consists of positive elements, by Theorem~\ref{something} the ideal $\I_S$ is directed.
As $X$ is Archimedean and pervasive, by Proposition~\ref{1.3.5} we conclude that $\sup M \in \I_S\dd= \B_S$. We have $a\in S^{\t{d}}  = S^{\t{ddd}} = \B_S^{\t{d}}$ and therefore $a\perp \sup M$.  The equality $\sup M=\sup S$ yields $a\perp \sup S$.
\end{proof}
Clearly, in the setting of Theorem~\ref{closedness.7supperp} we obtain for two sets $S_1,S_2\sub X_+$, for which the suprema $\sup S_1$ and $\sup S_2$ exist in $X$, that the relation $S_1\perp S_2$ implies $\sup S_1\perp \sup S_2$.

Note that in \c[Example~20]{02} an Archimedean pre-Riesz space $X$ and a directed band $B=\B_S$ for a set $S\sub B_+$ are given, where the existence of $\sup S$ in $X$ does not imply that $\sup S$ is an element of $B$. By Theorem~\ref{closedness.7supperp} we obtain conditions for an according affirmative statement.
%%%%%%%%%%%%%%%%%%%%%%%%%%%%%%%%%%%%%%%%%%%%%%%%%%%%%%%%%
%%%%				1.3.5 Theorem
%%%%%%%%%%%%%%%%%%%%%%%%%%%%%%%%%%%%%%%%%%%%%%%%%%%%%%%%%
\begin{corollary}\label{closedness.7supperpcor}
Let $X$ be an Archimedean pervasive pre-Riesz space and $S\sub X_+$ a non-empty subset for which $\sup S$ exists in $X$. Then $\sup S\in \B_S$.
\end{corollary}
%%%%%%%%%%%%%%%%%%%%%%%%%%%%%%%%%%%%%%%%%%%%%%%%%%%%%%%%%
%%%% END
%%%%%%%%%%%%%%%%%%%%%%%%%%%%%%%%%%%%%%%%%%%%%%%%%%%%%%%%%
\begin{proof}

For every $a\in S^{\t{d}}$ by Theorem~\ref{closedness.7supperp} it follows that $\sup S\perp a$. Hence by Lemma~\ref{0.0} we have $\sup S\in S\dd=\B_S$.
\end{proof}

%%%%%%%%%%%%%%%%%%%%%%%%%%%%%%%%%%%%%%%%%%%%%%%%%%%%%%%%%
%%%%
%%%%
%%%%		Order denseness in the extensions
%%%%
%%%%
%%%%%%%%%%%%%%%%%%%%%%%%%%%%%%%%%%%%%%%%%%%%%%%%%%%%%%%%%
\section{Order denseness of ideals and bands in their extensions}\label{extensions}

We consider again a pre-Riesz space $X$ and  a vector lattice cover $(Y,i)$ of $X$.
For a solvex ideal $I$ and a band $B$ in $X$, there exist the smallest extension ideal $\hat{I}$ of $I$ and the smallest extension band $\hat{B}$ of $B$ in $Y$. Clearly, these extensions are vector sublattices of $Y$. %Moreover, the linear and bipositive embedding of $X$ into $Y$ can be restricted to $I$ or to $B$, respectively, thereby inducing linear and bipositive mappings $i_1$ from $I$ into $\hat{I}$ and $i_2$ from $B$ into $\hat{B}$. That is, 
By means of $i$, we embed $I$ into the vector lattice $\hat{I}$, and $B$ into the vector lattice $\hat{B}$.
%every solvex ideal $I$ and every band $B$ in $X$ into a vector lattice. 
This raises the questions under which conditions $i(I)$ is order dense in $\hat{I}$, and $i(B)$ is order dense in $\hat{B}$.
%whether these embeddings $i_1$ and $i_2$ are order dense.
We give examples that these questions can not be  answered affirmatively, in general, even if  
$i(I)$ is majorizing in $\hat{I}$, or $i(B)$ is majorizing in $\hat{B}$.
We will show some affirmative results in pervasive pre-Riesz spaces.

We begin with an example of a pre-Riesz space that is not pervasive and provide a directed ideal $I$ such that $i(I)$ is not order dense in $\hat{I}$.
%%%%%%%%%%%%%%%%%%%%%%%%%%%%%%%%%%%%%%%%%%%%%%%%%%%%%%%%%
%%%%	BEGIN
%%%%%%%%%%%%%%%%%%%%%%%%%%%%%%%%%%%%%%%%%%%%%%%%%%%%%%%%%
\begin{example}\label{example1}
\textit{A directed ideal which is majorizing but not order dense in its smallest extension ideal.}

We resume Example~\ref{fordabilit_necessary}, where we considered $X:=\RR^3$ endowed with the order induced by the four ray cone $K_4$.
As is established in \c[Theorem~4.8 and Proposition~4.10]{7}, for every positive functional $f$ on $X$ which is extremal in the dual cone the kernel of $f$ is a solvex ideal in $X$. In particular, for the functional $f_4$ as in Example~\ref{fordabilit_necessary} the space $I:=\t{ker}f_4$ is a solvex ideal in $X$. Geometrically, $I$ corresponds to the plane generated by the vectors $v_1$ and $v_4$ in $K_4$. Note that $I$ is directed. For $S:=\left\{v_1,v_4\right\}$ we have $I=\I_S$. 
Applying Proposition~\ref{genideal=erwideal}, we obtain that the subspace $\hat{I}:=\I_{\left\{i(v_1),i(v_4)\right\}}=\left\{(x_1,x_2,x_3,0)^T\mid|x_1,x_2,x_3\in\RR\right\}$ of $Y=\RR^4$ is the smallest extension ideal of $I$. By Theorem~\ref{something} and Corollary~\ref{corsomething1} the subspace $i(I)$ is majorizing in $\hat{I}=\I_{i(I)}$.
We have
\[i(I)=i(X)\cap\hat{I}=\t{span}\left\{\smallmath[.6]{\left(\begin{array}{c} 1\\ 1\\ 0\\ 0 \end{array}\right)},\smallmath[.6]{\left(\begin{array}{c} 0\\ 1\\ 1\\ 0 \end{array}\right)}\right\}=\left\{\smallmath[.6]{\left(\begin{array}{c} \lambda\\ \lambda+\mu\\ \mu\\ 0 \end{array}\right)}\mid|\lambda,\mu\in\RR\right\}.\]
We show that $i(I)$ is not order dense in $\hat{I}$. To that end, let $z:=(1,0,1,0)\in\hat{I}$. We establish that $z$ can not be written as an infimum of elements in $i(I)$.
Let $M:=\left\{x\in i(I)\mid| z\leq x\right\} = \left\{(\lambda,\lambda+\mu,\mu,0)^T\mid| \lambda,\mu \geq 1\right\}$. As the lattice operations in $Y$ are pointwise, $\inf M$ exists in $i(X)$ and equals $(1,2,1,0)^T \neq z$.
That is, $i(I)$ is not order dense in $\hat{I}$.
\end{example}
%%%%%%%%%%%%%%%%%%%%%%%%%%%%%%%%%%%%%%%%%%%%%%%%%%%%%%%%%
%%%%	END
%%%%%%%%%%%%%%%%%%%%%%%%%%%%%%%%%%%%%%%%%%%%%%%%%%%%%%%%%

%%%%%%%%%%%%%%%%%%%%%%%%%%%%%%%%%%%%%%%%%%%%%%%%%%%%%%%%%
%%%%	BEGIN
%%%%%%%%%%%%%%%%%%%%%%%%%%%%%%%%%%%%%%%%%%%%%%%%%%%%%%%%%
\begin{numremark}
Note that Example~\ref{example1} yields a negative answer to the following question:
Let $X$ be an Archimedean pre-Riesz space with a vector lattice cover $(Y,i)$ and $D\sub X$ a directed subspace of $X$. Since $D$ is Archimedean, it follows that $D$ is pre-Riesz. Does there exist a sublattice $\hat{D}$ of $Y$ such that $i(D)$ is order dense in $\hat{D}$?
In Example~\ref{example1}, the subspace $i(I)$ of $Y$ is a vector lattice in the sense that it is stable under its own lattice operations, not those inherited from $Y$. That is, $i(I)$ is not a vector sublattice of $Y$.
If $A\sub Y$ is the smallest sublattice of $Y$ containing $I$, then $A$ has to contain the pointwise supremum of $(1,1,0,0)^T$ and $(0,1,1,0)^T$. Hence, $A=\hat{I}$. We already established that $i(I)$ is not order dense in $\hat{I}$. Therefore $Y$ does not contain a vector lattice cover of $I$.
\end{numremark}
%%%%%%%%%%%%%%%%%%%%%%%%%%%%%%%%%%%%%%%%%%%%%%%%%%%%%%%%%
%%%%	END
%%%%%%%%%%%%%%%%%%%%%%%%%%%%%%%%%%%%%%%%%%%%%%%%%%%%%%%%%

In Theorem~\ref{closedness.1001} below we give conditions such that a directed ideal is order dense in its smallest extension ideal.
First we need the following theorem which characterizes pervasive pre-Riesz spaces.
%%%%%%%%%%%%%%%%%%%%%%%%%%%%%%%%%%%%%%%%%%%%%%%%%%%%%%%%%
%%%%	BEGIN
%%%%%%%%%%%%%%%%%%%%%%%%%%%%%%%%%%%%%%%%%%%%%%%%%%%%%%%%%
\begin{theorem}\label{properties.28neu}
Let $X$ be an Archimedean pre-Riesz space, $Y$ a vector lattice and $i\colon X\rightarrow Y$ a bipositive linear map. Then the following statements are equivalent.
\begin{enumerate}
\item\label{properties.28neu.it1} $i(X)$ is order dense in $Y$ and $X$ is pervasive.
\item\label{properties.28neu.it2} $i(X)$ is majorizing in $Y$, and for every $y\in Y_+$ there is $A\sub i(X)_+$ such that $ y=\sup A$.
\end{enumerate}
\end{theorem}
%%%%%%%%%%%%%%%%%%%%%%%%%%%%%%%%%%%%%%%%%%%%%%%%%%%%%%%%%
%%%%	END
%%%%%%%%%%%%%%%%%%%%%%%%%%%%%%%%%%%%%%%%%%%%%%%%%%%%%%%%%
\begin{proof}
\ref{properties.28neu.it1} $\Rightarrow$ \ref{properties.28neu.it2}: Since $i(X)$ is order dense in $Y$, it is majorizing. For $y=0$ by setting $A:=\left\{0\right\}\sub i(X)$ we obtain $y=\sup A$. For $y\in Y_+$, $y\neq 0$, we can set $A:=\left\{v\in i(X)\mid| 0< v\leq y\right\}$. As $X$ is Archimedean, by Lemma~\ref{properties.1} we obtain $\sup A=y$.

\ref{properties.28neu.it2} $\Rightarrow$ \ref{properties.28neu.it1}: We first show the order denseness of $X$ in $Y$. To that end, we establish instead
\[\forall y\in Y\hs\exists A\sub i(X)\colon y=\sup A\]
and then apply Lemma~\ref{properties.29}. Clearly, if $y>0$, then by our hypothesis we are finished. Let $y\in Y$ be arbitrary. Then there exists an element $y_o\in Y_+$, $y_o\neq 0$, such that $y< y_o$. That is, we have $y=y_o-z$ for some $z\in Y_+$, $z\neq 0$.

We show first for the case $z\in i(X)$ that $y=y_o-z$ can be written as a supremum of a subset of $i(X)$. Since $y_o$ is positive, by assumption there exists a set $C\sub i(X)_+$ such that $y_o=\sup C$. It follows
$y = y_o-z = \sup C - z = \sup (C-z)$.
Due to $C\sub i(X)$ and $z\in i(X)$ 
we have $A:=C-z=\left\{c-z\mid|c\in C\right\}\subseteq i(X)$ and $y=\sup A$.

Let now $y=y_o-z$ with $z\in Y_+$, $z\neq 0$. Since $i(X)$ is majorizing in $Y$, there exists a $z\in i(X)$ such that $z\geq -y$, i.e.\ $y+z\geq 0$. Then the decomposition $y=(y+z)-z$ is such that $y_o:=y+z\in Y_+$ and $z\in i(X)$. By the above argumentation it follows that there exists a set $A\sub i(X)$ such that $y=\sup A$. Using Lemma~\ref{properties.29} we conclude that $X$ is order dense in $Y$. 

To see that $X$ is pervasive, notice that by assumption for every $y>0$ there exists a set $A\sub i(X)_+$ with $y=\sup A$. Clearly, $A\sub\left\{v\in i(X) \mid| 0\leq v\leq y\right\}$. By Lemma~\ref{properties.27.5} it follows that $\sup\left\{v\in i(X) \mid| 0\leq v\leq y\right\}$ exists in $X$ and equals $y$. With Lemma~\ref{properties.1} we conclude that $X$ is pervasive.
\end{proof}
%-%-%-%-%-%-%-%-%-%-%-%-%-%-%-%-%-%-%-%-%-%-%-%-%-%-%-%-

%%%%%%%%%%%%%%%%%%%%%%%%%%%%%%%%%%%%%%%%%%%%%%%%%%%%%%%%%
%%%%	BEGIN
%%%%%%%%%%%%%%%%%%%%%%%%%%%%%%%%%%%%%%%%%%%%%%%%%%%%%%%%%
\begin{theorem}\label{closedness.1001}
Let $X$ be an Archimedean pervasive pre-Riesz space with a vector lattice cover $(Y,i)$. Let $I\sub X$ be a directed ideal and $\hat{I}\sub Y$ the smallest extension ideal of $I$. Then $(\hat{I},i|_{I})$ is a vector lattice cover of the pre-Riesz space $I$, and $I$ is pervasive.
\end{theorem}
%%%%%%%%%%%%%%%%%%%%%%%%%%%%%%%%%%%%%%%%%%%%%%%%%%%%%%%%%
%%%%	END
%%%%%%%%%%%%%%%%%%%%%%%%%%%%%%%%%%%%%%%%%%%%%%%%%%%%%%%%%
\begin{proof}
Since the ideal $I$ is directed, it is solvex. Hence there exists the smallest extension ideal $\hat{I}$ of $I$ in $Y$. 
The ideal $\hat{I}$ is a sublattice of $Y$ with the order induced from $Y$.

We intend to show the properties of Theorem~\ref{properties.28neu}~\ref{properties.28neu.it2} for $I$ and $\hat{I}$. 
We first prove that $I$ is majorizing in $\hat{I}$.
%The smallest extension ideal $\hat{D}$ of $D$ is generated as the solvex hull of $i(D)$ in $Y$. In vector lattices the solvex hull of a linear subspace $S$ coincides with the ideal generated by $S$. 
As $I$ is directed, by Proposition~\ref{genideal=erwideal} and Corollary~\ref{corsomething1} we have $\hat{I}=\I_{i(I)}$.  Now Theorem~\ref{something} implies that $i(I)$ is majorizing in $\hat{I}$.

To show the second property of Theorem~\ref{properties.28neu}~\ref{properties.28neu.it2},
let $d\in \hat{I}_+$.
 Since $X$ is pervasive, by Lemma~\ref{properties.1} we obtain
\[d=\sup\left\{b\in i(X)\mid| 0< b\leq d\right\}.\]
As $\hat{I}$ is an extension ideal of $I$, we have $i(X)\cap\hat{I}=i(I)$. Using  the ideal property of $\hat{I}$, we get
\begin{align*}
d &=\sup\left\{b\in i(X)\mid| 0\leq b\leq d\right\} \\
&=\sup\left\{b\in i(X)\cap \hat{I}\mid| 0\leq b\leq d\right\} \\
&=\sup\left\{b\in i(I)\mid| 0\leq b\leq d\right\}.
\end{align*}
Thus the set $A:=\left\{b\in i(I)\mid| 0\leq b\leq d\right\}\sub i(I)_+$ is such that 
%That is, for every $d\in \hat{I}_+$ there exists a set $A\sub i(I)_+$ such that 
$\sup A=d$. 
Theorem~\ref{properties.28neu}
now implies that $i(I)$ is order dense in $\hat{I}$ and that  $I$ is pervasive.
\end{proof}

%%%%%%%%%%%%%%%%%%%%%%%%%%%%%%%%%%%%%%%%%%%%%%%%%%%%%%%%%
%%%%	BEGIN
%%%%%%%%%%%%%%%%%%%%%%%%%%%%%%%%%%%%%%%%%%%%%%%%%%%%%%%%%
\begin{corollary}\label{closedness.1001.cor1}
Let $X$ be an Archimedean pervasive pre-Riesz space. Let $I\sub X$ be a directed ideal and $\hat{I}$ the smallest extension ideal of $I$ in $X^\delta$. Then $(\hat{I},i|_{I})$ is the Dedekind completion of $I$, and $I$ is pervasive.
\end{corollary}
%%%%%%%%%%%%%%%%%%%%%%%%%%%%%%%%%%%%%%%%%%%%%%%%%%%%%%%%%
%%%%	END
%%%%%%%%%%%%%%%%%%%%%%%%%%%%%%%%%%%%%%%%%%%%%%%%%%%%%%%%%
\begin{proof}
We apply Theorem~\ref{closedness.1001} to $Y:=X^\delta$. Then Lemma~\ref{properties.0} yields the result.
\end{proof}

\begin{numremark}
A result analogous to Corollary~\ref{closedness.1001.cor1} for the Riesz completion $X^\rho$ does not exist, so far. For an Archimedean pervasive pre-Riesz space $X$, a directed ideal $I\sub X$ and the smallest extension ideal $\hat{I}\sub X^\rho$ of $I$, by Theorem~\ref{closedness.1001} the space $(\hat{I},i|_{I})$ is a vector lattice cover of $I$. We conjecture that $\hat{I}$ is the Riesz completion of $I$.
\end{numremark}

The next example shows that a directed band is not order dense in its smallest extension band, in general.
%%%%%%%%%%%%%%%%%%%%%%%%%%%%%%%%%%%%%%%%%%%%%%%%%%%%%%%%%
%%%%	BEGIN
%%%%%%%%%%%%%%%%%%%%%%%%%%%%%%%%%%%%%%%%%%%%%%%%%%%%%%%%%
\begin{example}\label{example2}
\textit{A directed band which is majorizing but not order dense in its smallest extension band.}

We resume Example~\ref{fordabilit_necessary}. % and \ref{example1}.
As is established in \c[Example~5.13]{2}, the linear subspace $B:=\t{span}\left\{v_2\right\}$ of $X$ is a band. The band $B$ corresponds in $i(X)$ to the linear subspace
$i(B)=\t{span}\left\{(0,0,1,1)^T\right\}$.
Since in $\RR^4$ the lattice operations are pointwise, the subspace $\hat{B}=\left\{(0,0,x_3,x_4)^T\mid| x_3,x_4\in\RR\right\}$ is the smallest extension band of $B$ in $\RR^4$.

Clearly, $i(B)$ is majorizing in $\hat{B}$. However, $i(B)$ is not order dense in $\hat{B}$. Let namely $z:=(0,0,0,1)^T\in\hat{B}$. By an argumentation similar to Example~\ref{example1} the element $z$ can not be written as an infimum of elements in $i(B)$.
\end{example}
%%%%%%%%%%%%%%%%%%%%%%%%%%%%%%%%%%%%%%%%%%%%%%%%%%%%%%%%%
%%%%	END
%%%%%%%%%%%%%%%%%%%%%%%%%%%%%%%%%%%%%%%%%%%%%%%%%%%%%%%%%

Next we study under which conditions a band has its smallest extension band as a vector lattice cover.
%%%%%%%%%%%%%%%%%%%%%%%%%%%%%%%%%%%%%%%%%%%%%%%%%%%%%%%%%
%%%%	BEGIN
%%%%%%%%%%%%%%%%%%%%%%%%%%%%%%%%%%%%%%%%%%%%%%%%%%%%%%%%%
\begin{theorem}\label{closedness.1002}
Let $X$ be an Archimedean pervasive pre-Riesz space with a vector lattice cover $(Y,i)$. Let $B\sub X$ be a band and $\hat{B}\sub Y$ an extension band of $B$. Then the following statements are equivalent.
\begin{enumerate}
\item\label{closedness.1002.it1} $i(B)$ is majorizing in $\hat{B}$.
\item\label{closedness.1002.it2} $i(B)$ is order dense in $\hat{B}$.
\end{enumerate}
Moreover, if \ref{closedness.1002.it1} or \ref{closedness.1002.it2} are satisfied, then $B$ is a pre-Riesz space and hence directed and $\hat{B}$ is a vector lattice cover of $B$. Moreover, $B$ is pervasive.
\end{theorem}
%%%%%%%%%%%%%%%%%%%%%%%%%%%%%%%%%%%%%%%%%%%%%%%%%%%%%%%%%
%%%%	END
%%%%%%%%%%%%%%%%%%%%%%%%%%%%%%%%%%%%%%%%%%%%%%%%%%%%%%%%%
\begin{proof}
The implication \ref{closedness.1002.it2} $\Rightarrow$ \ref{closedness.1002.it1} follows immediately.

We show \ref{closedness.1002.it1} $\Rightarrow$ \ref{closedness.1002.it2}. Since $i(B)$ is majorizing in the sublattice $\hat{B}$ of $Y$, it follows that $B$ is directed. Indeed, for $b_1,b_2\in B$ we have $i(b_1)\Sup i(b_2)\in \hat{B}$. As $B$ is majorizing in $\hat{B}$, there exists a $c\in B$ with $i(b_1)\Sup i(b_2) \leq i(c)$, i.e.\ $b_1,b_2\leq c$. Due to $X$ being pervasive and Archimedean Lemma~\ref{properties.1} implies for every $d\in \hat{B}_+$
\[d=\sup\left\{b\in i(X)\mid| 0\leq b\leq d\right\}.\]
For the extension band $\hat{B}$ we have $i(X)\cap \hat{B}=i(B)$. The ideal property of $\hat{B}$ yields
\begin{align*}
d &=\sup\left\{b\in i(X)\mid| 0\leq b\leq d\right\} =\\
&=\sup\left\{b\in i(X)\cap \hat{B}\mid| 0\leq b\leq d\right\} =\\
&=\sup\left\{b\in i(B)\mid| 0\leq b\leq d\right\}.
\end{align*}
That is, for every $d\in \hat{B}_+$ there exists a set $A\sub i(B)$ such that $\sup A=d$. As $i(B)$ is majorizing in $\hat{B}$, by Theorem~\ref{properties.28neu} we obtain that $i(B)$ is order dense in $\hat{B}$ and $B$ is pervasive.
\end{proof}

%%%%%%%%%%%%%%%%%%%%%%%%%%%%%%%%%%%%%%%%%%%%%%%%%%%%%%%%%
%%%%	BEGIN
%%%%%%%%%%%%%%%%%%%%%%%%%%%%%%%%%%%%%%%%%%%%%%%%%%%%%%%%%
\begin{corollary}\label{closedness1002.cor1}
Let $X$ be an Archimedean pervasive pre-Riesz space. Let $B$ be a band in $X$, $\hat{B}$ the smallest extension band of $B$ in $X^\delta$, and let $i(B)$ be majorizing in $\hat{B}$. Then $(\hat{B},i|_{B})$ is the Dedekind completion of $B$, and $B$ is pervasive.
\end{corollary}
%%%%%%%%%%%%%%%%%%%%%%%%%%%%%%%%%%%%%%%%%%%%%%%%%%%%%%%%%
%%%%	END
%%%%%%%%%%%%%%%%%%%%%%%%%%%%%%%%%%%%%%%%%%%%%%%%%%%%%%%%%

Let $X$ be an Archimedean pervasive pre-Riesz space with a vector lattice cover $(Y,i)$.
Clearly, if $\hat{B}$ is a vector lattice cover of a band $B$ in $X$, then $B$ is directed. Is the converse true? That is, is an analogous result to Theorem~\ref{closedness.1001} true for bands? This is not the case, as in Example~\ref{ex2} the band $B$ is directed in an Archimedean pervasive pre-Riesz space $X$, but $i(B)$ is not majorizing in $\hat{B}$. Thus by Theorem~\ref{closedness.1002} the subspace $i(B)$ is not order dense in $\hat{B}$.
Though $\hat{B}$ in Example~\ref{ex2} is not a proper candidate for a vector lattice cover of $B$, the smallest extension ideal $\hat{I}$ of $B$ in $Y$ is. Indeed, as $B$ is a directed ideal, $\hat{I}$ exists
and by Proposition~\ref{genideal=erwideal} and Corollary~\ref{corsomething1} we have $\hat{I}=\I_{i(B)}=\left\{f\in Y \mid| f([0,\tfrac{1}{2}]) =0,\hs f \t{ is differentiable in } \tfrac{1}{2}\right\}$. By Theorem~\ref{closedness.1001} the ideal $\hat{I}$ is a vector lattice cover of $B$.

In this example $\hat{B}$ and $\hat{I}$ do not coincide. This raises the following questions: Are there instances where the smallest extension ideal and the smallest extension band of a band coincide? Is this property dependent on the vector lattice cover in consideration?
The subsequent example, which originated during a discussion with O.~van~Gaans, gives affirmative answers to those questions.
%%%%%%%%%%%%%%%%%%%%%%%%%%%%%%%%%%%%%%%%%%%%%%%%%%%%%%%%%
%%%% BEGIN
%%%%%%%%%%%%%%%%%%%%%%%%%%%%%%%%%%%%%%%%%%%%%%%%%%%%%%%%%
\begin{example}\label{closedness.18}
\textit{The smallest extension ideal and the smallest extension band of a band in a pervasive pre-Riesz space might or might not coincide, depending on the considered vector lattice cover.}

Let 
$\t{PA}[-1,1] := \left\{f\in C[-1,1] \mid| f \t{ piecewise affine} \right\}$
be the vector lattice of all piecewise affine continuous functions on $[-1,1]$ (finitely many pieces). Denote by $U(0)\sub[-1,1]$ a neighbourhood of $0$ and let
\[X_0:=\left\{f\in \t{PA}[-1,1] \mid| \exists U(0), c\in\RR \hs\forall t\in U(0):\hs f(t) = c\right\}.\]
Then $X_0$ is a vector lattice which is order dense in $Y:= C[-1,1]$.
%\textit{order dense in the vector lattice sense}\footnote{By the classical definition in \c[§1.3]{PosOp} a vector sublattice $G$ of a vector lattice $E$ is called \textit{order dense (in the vector lattice sense)} whenever for each $x\in E_+$, $x\neq 0$, there exists some $y\in G$ with $0<y\leq x$. Note that this notion does \textbf{not} coincide with our notion \textit{order dense} from \c[p.~360]{113}, which is widely used throughout this paper. Note also the similarities to the term pervasive, which differs from order denseness in the vector lattice sense by being a property of a pre-Riesz space in a vector lattice cover.}
We can regard $X_0$ as a pre-Riesz space with $X_0$ as its own Riesz completion, hence $X_0$ is pervasive. 

Define a function $q\in C[-1,1]$ by
\[q(t)= \begin{cases}t^2-\frac{1}{4} &\quad t\in\left[-1,-\frac{1}{2}\right]\cup\left[\frac{1}{2},1\right]\\ 0 &\quad t\in \hs\left]-\frac{1}{2},\frac{1}{2}\right[. \end{cases}\]
The ordered vector space $X:= \t{span}\left(X_0 \cup \left\{\lambda q \in C[-1,1]\mid| \lambda\in\RR\right\}\right) \sub Y$ is a pre-Riesz space. However, $X$ is not a vector lattice, since $X$ does not have the RDP. Indeed, define two functions $a_1,a_2\in X$ with $a_1(t)=-t-\tfrac{1}{2}$ for $t\in[-1,-\tfrac{1}{2}]$ and $a_1(t)=0$ for $t\in\hs]-\tfrac{1}{2},1]$ and $a_2(t)=t-\tfrac{1}{2}$ for $t\in[\tfrac{1}{2},1]$ and $a_2(t)=0$ for $t\in[-1,\tfrac{1}{2}[$. Then $q\leq a_1+a_2$. However, there do not exist $q_1,q_2\in X$ such that $q=q_1+q_2$ and $q_1\leq a_1$, $q_2\leq a_2$.

Since $X_0\sub X$, we obtain that $X$ is order dense in $Y$ and pervasive. As the second vector lattice cover of $X$ we consider the Riesz completion $X^\rho$, which can be represented as the vector sublattice of $Y$ generated by $X$. Note that every element of $X^\rho$ is constant on a neighbourhood of $0$.
The set 
\[B:=\left\{f\in X\mid|\exists U(0) \hs\forall t\in [-1,0]\cup U(0):\hs f(t) = 0\right\}\]
is a band in $X$. Observe that $B\sub X_0$. As $B$ is a band in the vector lattice $X_0$, it is directed.
We show that the smallest extension ideal and the smallest extension band of $B$ in $X^\rho$ coincide, whereas they do not coincide in the vector lattice cover $Y$.

The subspace
\[\I_B^Y:=\left\{f\in Y  \mid| \exists U(0) \hs\forall t\in [-1,0]\cup U(0):\hs f(t) = 0\right\}\]
is the ideal generated by $B$ in $Y$. As $B$ is solvex, $\I_B^Y$ is the smallest extension ideal of $B$ in $Y$. Analogously, the ideal $\I_B^{X^\rho}$ generated by $B$ in $X^\rho$ is the smallest extension ideal of $B$ in $X^\rho$.

The ideal $\I_B^Y$ is not a band. Indeed, the function $x\in (\I_B^Y)\dd$ with $x(t)=0$ for $t\in [-1,0]$ and $x(t)=t$ for $t\in\hs]0,1]$ does not belong to $\I_B^Y$. In the vector lattice $Y$ we have $(\I_B^Y)\dd=\B_B^Y$, which is the smallest extension band of $B$ in $Y$. Thus we established $\I_B^Y\hs\subsetneq\hs \B_{B}^Y$.

For $\I_B^{X^\rho}$ we have $(\I_B^{X^\rho})^\t{d}=\left\{f\in X^\rho  \mid| f([0,1]) = 0\right\}$. Since every element of $X^\rho$ is constant on a neighbourhood of $0$, we obtain
\[(\I_B^{X^\rho})\dd=\left\{f\in X^\rho  \mid| \exists U(0) \hs\forall t\in [-1,0]\cup U(0):\hs f(t) = 0\right\} = \I_B^{X^\rho}.\]
In the vector lattice $X^\rho$ the band $(\I_B^{X^\rho})\dd=\B_B^{X^\rho}$ is the smallest extension band of $B$. So we have shown $\I_B^{X^\rho}=\B_B^{X^\rho}$.

The example shows that the properties of extension bands and extension ideals might change, depending on which vector lattice cover we deal with. 
\end{example}

To sum up, the main results in the Theorems \ref{closedness.1001} and \ref{closedness.1002}
are set in pervasive pre-Riesz spaces. It is an interesting question for further research whether similar statements are true under milder conditions.
%%%%%%%%%%%%%%%%%%%%%%%%%%%%%%%%%%%%%%%%%%%%%%%%%%%%%%%%%
%%%% END
%%%%%%%%%%%%%%%%%%%%%%%%%%%%%%%%%%%%%%%%%%%%%%%%%%%%%%%%%

%%%%%%%%%%%%%%%%%%%%%%%%%%%%%%%%%%%%%%%%%%%%%%%%%%%%%%%%%
%%%%
%%%%				END CONTENT
%%%%
%%%%%%%%%%%%%%%%%%%%%%%%%%%%%%%%%%%%%%%%%%%%%%%%%%%%%%%%%

%%\newpage
%\bibliographystyle{plain}
%\bibliography{lit.bib}
% Publisher does not run BibTeX. Therefore here are the contents of the .bbl file after runnung BibTeX with the above two commands uncommented:

\end{document}